\documentclass[journal,twoside,web]{ieeecolor}
\usepackage{lcsys}

\usepackage{graphicx} 

\usepackage{amsmath} 
\usepackage{amssymb}  
\usepackage{mathrsfs}

\usepackage{amsthm}
\usepackage{mathtools}
\usepackage{bm}
\usepackage{algorithm}
\usepackage{cite}
\usepackage{algpseudocode}[noend]
\usepackage{gensymb}
\usepackage{soul}

\usepackage[linkcolor=blue,colorlinks=true]{hyperref}

\theoremstyle{definition}
\newtheorem{definition}{Definition}
\newtheorem{theorem}{Theorem}
\newtheorem{lemma}{Lemma}

\newtheorem{remark}{Remark}

\newtheorem{assumption}{Assumption}

\newtheorem{example}{Example}

\usepackage{subcaption}
\usepackage{verbatim}

\newtheorem{objective}{Objective}

\usepackage{caption}
\DeclareCaptionType{equ}[][]

\graphicspath{figures/} 

\title{\LARGE \bf
Rate-Tunable Control Barrier Functions: \\Methods and Algorithms for Online Adaptation
}

\author{Hardik Parwana, \IEEEmembership{Member, IEEE}, Dimitra Panagou, \IEEEmembership{Senior Member, IEEE}
\thanks{Manuscript submitted on March 20, 2023. This work was partially sponsored by the Office of Naval Research (ONR), under grant number N00014-20-1-2395. The views and conclusions contained herein are those of the authors only and should not be interpreted as representing those of ONR, the U.S. Navy or the U.S. Government.}
\thanks{Hardik Parwana and Dimitra Panagou are with Department of Robotics, University of Michigan, MI 48105, USA (email: hardiksp@umich.edu, dpanagou@umai.edu). }}

\newcommand{\reals}{\mathbb{R}}
\newcommand{\s}{\mathcal{S}}
\newcommand{\X}{\mathcal{X}}
\newcommand{\U}{\mathcal{U}}

\newcommand{\K}{\mathcal{K}}

\newcommand{\C}{\mathcal{C}}
\newcommand{\V}{\mathcal{V}}
\newcommand{\classK}{\mbox{class-$\K$} }

\newcommand{\Int}{\text{Int} }

\newcommand{\eqn}[1]{\begin{align}#1\end{align}}

\newcommand{\red}[1]{\textcolor{red}{#1}}

\begin{document}

\maketitle
\thispagestyle{empty}
\pagestyle{empty}


\begin{abstract}
Control Barrier Functions offer safety certificates by dictating controllers that enforce safety constraints. However, their response depends on the \classK function that is used to restrict the rate of change of the value of the barrier function along the system trajectories. This paper introduces the notion of a Rate-Tunable (RT) CBF, which allows for online tuning of the response of CBF-based controllers. In contrast to existing approaches that use a fixed \classK function to ensure safety, we parameterize and adapt the \classK function parameters online. We show that this helps improve the system's response in terms of the resulting trajectories being closer to a nominal reference while being sufficiently far from the boundary of the safe set. We provide point-wise sufficient conditions to be imposed on any user-given parameter dynamics so that multiple CBF constraints continue to admit a common control input with time. Finally, we introduce RT-CBF parameter dynamics for decentralized noncooperative multi-agent systems, where a trust factor, computed based on the instantaneous ease of constraint satisfaction, is used to update parameters online for a less conservative response.

\end{abstract}

\section{Introduction \& Previous Works}


The notion of a control barrier function (CBF) has emerged as a tool to ensure constraint satisfaction for dynamical systems. The principle is as follows: Given a constraint function, termed barrier function thereafter, whose zero super-level sets define a constrained set, termed also the safe set, 
one restricts the rate of change of the barrier function along the system trajectories using a $\classK$ function of the barrier function \cite{ames2016control, ames2019control}. If such a condition can be satisfied everywhere in the constrained set under the given dynamics and control input constraints, then the barrier function is called a CBF, and the constrained set is forward invariant. This method, in conjunction with control Lyapunov functions (CLFs) for stability, has been employed to design safe controllers for several applications.

\subsection{Background and open challenges}
Despite significant extensions of the CBF concept to address various challenges (e.g., modeling/parameter uncertainty, input constraints) over the last few years, we highlight three unsolved issues. 

\subsubsection{Finding valid CBFs} Given a single constraint function, finding a suitable (valid) CBF reduces in effect to finding both a barrier function and a $\classK$ function so that the CBF condition can be realized for given dynamics and control input bounds. This task is non-trivial and several approaches, mostly offline, have been proposed recently\cite{dawson2023safe,jagtap2020control}. When the safe set is represented as an intersection of the sets defined by several constraints, for example, when there are multiple unsafe regions such as obstacles, the objective becomes finding a controller that simultaneously ensures the invariance of the intersection of the sets. Existing CBF implementations either combine all barriers into a single barrier function \cite{glotfelter2017nonsmooth, stipanovic2012monotone, panagou2013multi} resulting in a single CBF derivative condition, or consider each barrier separately\cite{usevitch2020strong,xiao2020feasibility,wang2022ensuring} thereby imposing multiple CBF derivative conditions simultaneously. Either case only adds to the complexity of finding valid CBFs. Furthermore, imposing multiple CBFs and control input bounds in the optimization problem can lead to infeasibility of the solution (controller) at the current or at some future time. 
We call this problem ``lack of persistent feasibility" of the low-level CBF controller. Note that a good choice of CBF parameters may lead to persistent feasibility, whereas a poor choice may not; however, finding a good choice of parameters is not a trivial problem. 

\subsubsection{High relative degree}
The notion of CBFs as first introduced in \cite{ames2016control} applies to constraints of relative degree $r=1$ with respect to the system dynamics, i.e., to constraints whose first-order derivative is a function of the control input. The concept is extended to constraint functions of higher relative degree $r>1$, i.e., to the case where the control input appears in the $r$-th -order derivative of the constraint function. \cite{nguyen2016exponential} proposed the class of Exponential CBFs when the $\classK$ functions used in CBF derivative condition are linear in their argument. They employ input-output linearization and design controllers via pole placement so that the barrier state is stabilized to zero. \cite{xiao2019control} generalizes Exponential CBFs to generic nonlinear $\classK$ functions in the form of Higher-Order CBFs (HOCBF). \cite{taylor2022safe} proposes an alternate Lyapunov backstepping-based design for higher-order systems that always find a first-order barrier function. 

\subsection{Current approaches and limitations}
\label{section::review_limitations}
\subsubsection{Finding valid CBFs} In recognition of the importance of finding a valid control barrier function, several works have also attempted to learn a barrier function, or to find the domain over which a given barrier function satisfies the CBF condition. \cite{agrawal2021safe} does offline analysis to find the domain over which CBF condition can be satisfied in the presence of input bounds. \cite{robey2020learning, lindemann2021learning} perform offline analysis to find a barrier function that satisfies the CBF derivative condition. A more thorough survey on learning-based methods can be found in \cite{dawson2022safe}. 

When multiple constraints are present, \cite{stipanovic2012monotone, panagou2013multi} combine barrier functions into a single barrier constraint through smoothed minimum (or maximum) operator. \cite{black2022adaptation} proposes a novel weighted-summation technique, where the weights are tuned for performance and to avoid singular states where the control input has no contribution to barrier function derivatives (usually the latter is only assumed, but not guaranteed, to not happen in the literature). The authors in \cite{glotfelter2017nonsmooth} perform non-smooth analysis to construct a barrier function with the non-smooth minimum (or maximum) operation. The aforementioned approaches thus either involve smoothing approximations or require intricate non-smooth analysis. When multiple CBF derivative conditions are considered simultaneously instead as in 
\cite{usevitch2020strong, xiao2020feasibility,wang2022ensuring},  \cite{tan2022compatibility} proposes an offline sampling-and-grid refinement method to search the domain of state space over which multiple control barrier functions are compatible.
Multiple CBF constraints also arise naturally in multi-agent scenarios. \cite{lindemann2019control,lindemann2020barrier} encode Signal Temporal Logic (STL) tasks and safety specifications as CBF constraints. \cite{borrmann2015control, usevitch2021adversarial} propose CBFs for safe swarm behavior in the presence of uncooperative agents when their identities are known in advance. These uncooperative agents serve as an excellent example of uncontrollable systems, where different CBF $\classK$ functions are suited to controllable and uncontrollable dynamics. The safe sets are also time-varying for the controllable dynamics, which 
most of the aforementioned papers have not considered in their theoretical analysis. \cite{usevitch2021adversarial} designs a robust CBF controller for systems with adversarial agents without detecting malicious agents. \cite{mustafa2022adversary} detects malicious agents by checking for collisions in the predicted future wherein the ego and target agents, forgoing their other objectives, try to only maintain and violate respectively, the CBF derivative constraint. Both the above-mentioned approaches are conservative due to the consideration of worst-case bounds. 

\subsubsection{Feasibility} All of the aforementioned works use a constant $\classK$ function over the whole domain, and its parameters are often pre-tuned manually before deploying for an actual application. However, since it is not possible to test each and every scenario in the offline phase, there is a dire need for online adaptation, especially to maintain the feasibility of the controller in the presence of multiple constraints and input bounds.
Several works relax the CBF constraint in the QP to ensure that the QP controller has a solution. \cite{zeng2021safety} proposes an optimal-decay CBF-QP that multiplies the $\classK$ function by a scalar variable that is optimized for (along with the control input) in the QP. \cite{wang2022ensuring}, \cite{garg2019control} introduced similar relaxation strategies for CBF constraints with the subtle difference being that instead of linear multiplicative factor to generic $\classK$ function, they only consider linear $\classK$ functions and treat the $\classK$ function parameter as an optimization variable. \cite{seo2022safety} considers control non-affine dynamics and in light of the resulting loss of linearity of the CBF constraint, chooses to impose it through a penalty term in the cost function, thereby converting it into an unconstrained optimization problem.


The aforementioned adaptive approaches can be thought of as adapting \textit{on a need basis}; i.e., they relax the CBF constraint whenever the original $\classK$ function leads to infeasibility. However, persistent feasibility is not guaranteed for multiple CBF constraints in the presence of input bounds. Since they also solve for the control input point-wise (via a QP), they are still sub-optimal in the long term. Among the non-myopic approaches, \cite{xiao2020feasibility} multiplies a scalar penalty to the $\classK$ functions, and creates a dataset of persistently feasible and eventually infeasible parameters offline by randomly sampling parameters and simulating the system forward. Then they fit a differentiable classifier that is used to impose the resulting feasibility margin as another HOBCF constraint in the system. \cite{ma2022learning} performs offline learning to learn the best-performing parameters of the $\classK$ function as a function of the state, although without consideration of feasibility. 
\subsection{Contributions}

Building upon our previous work \cite{parwana2022trust}, we introduce a new notion of a Rate-Tunable Control Barrier Function (RT-CBF) in Section \ref{section::RT-CBF} that allows for a time-varying $\classK$ function while ensuring safety.
In practice, this allows consideration of parametric $\classK$ functions and adaptation of their parameters online. This adaptation can be used to make the controller less or more conservative without jeopardizing safety.

 It is also noteworthy that this adaptation facilitates the consideration and satisfaction of multiple time-varying barrier constraints, by making them easier to tune for performance, especially when they do not represent similar physical quantities (e.g., when imposing constraints on the rotational dynamics, and constraints on the translational dynamics for a quadrotor). In Section \ref{section::problem_formulation}, we pose the problem of guaranteeing the existence of a solution to CBF-QP controller as persistently ensuring \textit{compatibility} between multiple CBF constraints along system trajectories. We then establish the necessity and sufficiency of our notion of safety in Section \ref{section::RT-CBF}. 
Designing the parameter dynamics however is non-trivial, especially in the presence of multiple constraints, for myopic controllers. We introduce pointwise sufficient conditions on the rate of change of parameters for enforcing feasibility in Section \ref{section::suboptimal_parameter_dynamics}. Although the pointwise design is suboptimal it improves upon the CBF-QP controllers. It also allows the incorporation of any
user-designed rule, even heuristic, for updating the $\classK$ function, and project it to a set of feasible update rules using our sufficient conditions to ensure feasibility whenever these conditions are well defined. Finally, as a case study, we design RT-CBFs for decentralized control for multi-systems in Section \ref{section::multi-agent} based on our previous work in \cite{parwana2022trust}. Specifically, we design the parameter dynamics based on a trust factor that is computed based on the instantaneous ease of satisfaction of the CBF constraints, and illustrate how this can be applied to robots of heterogeneous dynamics. 

\subsection{Comparison to Earlier Work}
The work most similar to our notion of CBF is \cite{xiao2021adaptive}, wherein the authors introduce a time-varying \textit{penalty parameter} that multiplies to the fixed $\classK$ function. The penalty parameter is allowed to change with time, but is constrained to be non-negative and hence is itself posed as higher-order CBF. The approach proposes another low-level optimization problem that finds the penalty parameter closest to a nominal parameter value subject to HOCBF constraints. However, this work does not consider multiple barrier functions simultaneously. And while it does consider control input bounds, the proof for persistent feasibility (existence of a solution to HOCBF QP for all future times) is based on the assumption that the original MPC-like formulation always admits a solution irrespective of the initial state, which is possible only because the environment (and hence the constraints) is time-invariant. 

Compared to our own previous work in \cite{parwana2022trust}, we make the following new contributions: (1) Modification of the definition of RT-CBFs to represent a weaker safety condition with more practical importance. (2) Proof of necessity and sufficiency of RT-CBFs, and finally, (3) Consideration of control input bounds and constraints of higher relative degree with respect to the system dynamics and evaluation of our methods on several new simulation case studies, including multi-agent systems with higher-order barrier functions. 

The rest of the paper is structured as follows: our notations and some preliminary information is discussed in Sections \ref{section::notation} and \ref{section::preliminaries}. Section \ref{section::problem_formulation} formulates our problem and Section \ref{section::RT-CBF} introduces our notion of RT-CBFs. We then propose pointwise sufficient conditions on \classK function parameter dynamics for persistent feasibility in Section \ref{section::dynamics_design} and introduce RT-CBFs for multi-agent systems in Section \ref{section::multi-agent}. Finally, we present simulation results in Section \ref{section::simulation} and conclude the paper in Section \ref{section::conclusion}.

\section{Notations}
\label{section::notation}
The set of real numbers is denoted as $\reals$ and the non-negative real numbers as $\reals^+$. Given $x\in \reals$, $y \in \reals^{n_i}$, and $z\in \reals^{n_i\times m_i}$, $|x|$ denotes the absolute value of $x$ and $||y||$ denotes $L_2$ norm of $y$. The interior and boundary of a set $\C$ are denoted by $\textrm{Int}(\C)$ and $\partial \C$. For $a\in \reals^+$, a continuous function $\alpha:[0,a)\rightarrow[0,\infty)$ is a $\classK$ function if it is strictly increasing and $\alpha(0)=0$.  A continuous function $\alpha:(-b,a)\rightarrow (-\infty,\infty)$ for $a,b\in \reals^+$ is an extended $\classK$ function if it is strictly increasing and $\alpha(0)=0$. Furthermore, if $a=\infty$ and $\lim_{r\rightarrow \infty} \alpha(r)=\infty$, then it is called extended class-$\mathcal{K}_\infty$. 
The $k^{th}$ time derivative of a function $h(t,x):\reals^{+}\times \reals^{n}\rightarrow \reals$ is denoted as $h^{(k)}$. 
For brevity, we will refrain from mentioning explicit arguments whenever the context is clear. For example, $h(x)$ may simply be denoted as $h$. The Lie derivative of a function $h$ w.r.t a function $f$ is denoted as $L_f h = \frac{\partial h}{\partial x}f$. The angle between $\gamma$ two vectors $a,b\in \reals^n$ is given by $\cos \gamma=\frac{a^T b}{||a||~||b||}$ and the inner product between them is denoted as $\langle a, b \rangle$.

\section{Preliminaries}
\label{section::preliminaries}
\subsubsection{System Description}
Consider the nonlinear control-affine dynamical system:
\eqn{
\dot{x} = f(x) + g(x)u,
\label{eq::dynamics_general}
}
where $x\in \X \subset \reals^{n}$ and $u \in \U \subset \reals^{m}$ represent the state and control input, and $f:\X\rightarrow \reals^{n}$ and $g:\X\rightarrow \reals^{m}$ are locally Lipschitz continuous functions. The set $\s(t)$ of allowable states at time $t$ is specified as an intersection of $N$ sets $\mathcal S_i(t),i\in\{1,2,..,N\}$, each of which is defined as the zero-superlevel set of a (sufficiently smooth) function $h_i:\reals^+ \times \mathcal{X} \rightarrow \reals$ as:
\begin{subequations}
    \begin{align}
        \s_i(t) & \triangleq \{ x \in \X : h_i(t,x) \geq 0 \}, \label{eq::safeset1} \\
        \partial \s_i(t) & \triangleq \{ x\in \X: h_i(t,x)=0 \}, \label{eq::safeset2}\\
        \Int (\s_i)(t) & \triangleq \{ x \in \X: h_i(t,x)>0  \}. \label{eq::safeset3}
    \end{align}
    \label{eq::safeset}
\end{subequations}
we also define a function  $x:\reals^+\rightarrow \mathcal{X}$ to denote the state $x(t)$ at time $t$. Note that $x(t)$ may be either an arbitrarily defined state trajectory, or may result from a closed-loop system response under a specific controller; which case is considered will be explicitly stated. Similarly the function $u: \reals^+\rightarrow\reals^m$ defines a control input trajectory.
\subsubsection{Control Barrier Functions}\footnote{Definitions \ref{definition::cbf_definition} and \ref{definition::hocbf_definition} were presented in their original papers for the time-invariant safe sets $\s_i$. We note that an extension to the time-varying case can be proven with Nagumo's theorem applied to non-autonomous systems \cite[Theorem 3.5.2]{carja2007viability} and hence we directly present that. This follows also the notation in \cite{lindemann2018control}. }

\begin{definition}
\label{definition::cbf_definition}
(Control Barrier Function) \cite{ames2016control} For the dynamical system (\ref{eq::dynamics_general}), $h_i : \mathbb R^+ \times \mathcal X\rightarrow \mathbb R$ is a control barrier function (CBF) on the set $\s_i(t)$ defined by (\ref{eq::safeset1})-(\ref{eq::safeset3}) for $t\geq 0$, if there exists a class-$\mathcal{K}$ function $\alpha_i : \reals \rightarrow \reals^+$ 
such that 
\eqn{
 \sup_{u\in \mathcal{U}} \left[ \frac{\partial h_i(t,x)}{\partial t} + L_f h_i(t,x)+ L_g h_i(t,x)u \right] \geq -\alpha_i(h_i(t,x)) \nonumber \\
 \forall x\in \mathcal{S}_i, \forall t>0.
  \label{eq::cbf_derivative}
}
\end{definition}
Henceforth, we refer to \eqref{eq::cbf_derivative} as the \textit{CBF derivative condition}.

\begin{theorem}(Set Invariance) \cite{ames2016control}
Given the dynamical system (\ref{eq::dynamics_general}) and a set $\s_i(t)$ defined by (\ref{eq::safeset1})-(\ref{eq::safeset3}) for some continuously differentiable function $h_i:\reals^+ \times \reals^n\rightarrow \reals$, if $h_i$ is a control barrier function on the set 
$\s_i(t)$, then $\s_i(t)$ is forward invariant.
\end{theorem}

When $L_g h_i(t,x) \equiv 0$ $\forall (t,x)\in\mathbb R^+\times \mathcal X$, then the control input $u$ does not appear in the left-hand side of the CBF condition \eqref{eq::cbf_derivative}. Suppose the relative degree of the function $h_i$ w.r.t. the control input $u$ under the dynamics \eqref{eq::dynamics_general} is equal to $r_i\geq 2$. We can then define $r_i$ functions as follows:
\begin{subequations}
    \eqn{
   \psi_i^0(t,x) &= h_i(t,x), \\
   \psi_i^k(t,x) &= \dot \psi_i^{k-1}(t,x) + \alpha_{i}^k (\psi^{k-1}_i(t,x)), \; k\in \{1,2,\dots,r_i-1\} \label{eq::hocbf_derived_barriers},
}
\label{eq::hobcf_barriers}
\end{subequations}
and denote their zero-superlevel sets respectively, as: 
\eqn{
  \mathscr{C}_i(t) &= \{x \; | \; \psi_i^k(t,x)\geq 0, \quad \forall k\in \{0,..,r_i-1\} \}.
  \label{eq::hobcf_intersection}
}
\begin{definition}(HOCBF CBF)\cite{tan2021high}
\label{definition::hocbf_definition}
    The function $h_i(t,x):\reals^+\times \reals^{n}\rightarrow \reals$ is a Higher-Order CBF (HOCBF) of $r_i$-th order on the set $\mathscr{C}_i(t)$ if there exist $r_i$ extended $\classK$ functions $\alpha_i^k:\reals \rightarrow \reals, \;  k\in\{1,2,..,r_i\}$, and an open set $\mathscr{D}_i(t)\subset \reals^+ \times \reals^n$ with $\mathscr{C}_i(t) \subset \mathscr{D}_i(t) \subset \mathcal{X}$ such that
     \eqn{
      \dot \psi_i^{r_i-1}(t,x,u) \geq -\alpha_i^{r_i} (\psi_i^{r_i-1}(t,x)), ~ \forall x \in \mathscr{D}_i(t), \forall t\geq 0.
      \label{eq::hobcf_cbf_derivatrive}
    }
\end{definition}
In this work, for each constraint function $h_i$, we consider $r_i$ \emph{parametric} $\classK$ functions $\alpha_i^k: \reals \times \mathcal A_i^k \rightarrow \reals^+$, $k\in\{1,\dots,r_i\}$, where $\mathcal A_i^k\subset \mathbb R^{n_{\alpha_i}^k}$, where $n_{\alpha_i}^k$ is the number of parameters of the function $\alpha_i^k$. Denote the parameters (of the function $\alpha_i^k$) by the vector $\theta_{\alpha_i^k}\in \mathcal A_i^k$. For example, a linear $\classK$ function 
\eqn{
  \alpha_i^k(\nu_i^k,h_i) = \nu_i^k h_i, \; \nu_i^k \in \reals
  \label{eq::linear_classK_example}
}
is described by its constant parameter $\theta_{\alpha_i^k} = \nu_i^k$. The $\classK$ function may then be described as $\alpha_i^k(\theta_{\alpha_i^k},h_i)$. 

\begin{remark}For brevity, whenever clear from the context, in the sequel we drop $\theta_{\alpha_i^k}$ as an argument of $\alpha_i^k(\theta_{\alpha_i^k})$, and denote the function as $\alpha_i^k$.
\end{remark}

The (column) vector concatenation of the parameters of all the $r_i$ \classK functions $\alpha_i^k$, $k\in\{1,\dots,r_i\}$ of the $i$-th barrier function $h_i(t,x)$ is denoted as $\theta_{\alpha_i}=\begin{bmatrix}\theta_{\alpha_i^1}^T \dots \theta_{\alpha_i^{r_i}}^T\end{bmatrix}^T$, and the (column) vector of all the parameters of all $\classK$ functions $\alpha_i^k, i\in \{1,..,N\}, k\in \{1,..,r_i\}$ is denoted as $\theta_{\alpha} = \begin{bmatrix}\theta_{\alpha_1}^T & \dots & \theta_{\alpha_N}^T\end{bmatrix}^T$. We denote the size of vectors $\theta_{\alpha_i}$ and $\theta_{\alpha}$ as $n_{\alpha_i}$, and $n_\alpha$, respectively.

In the ensuing sections, we allow the parameters $\theta_{\alpha_i}$ to change by designing their dynamics $\theta_{\alpha_i}^{(p)}$, where $z^{(p)}$ denotes the time derivative of order $p$. The conditions  \eqref{eq::hocbf_derived_barriers} and \eqref{eq::hobcf_cbf_derivatrive} then also depends on the derivatives of $\theta_{\alpha_i}$.
For example, for linear \classK functions as in \eqref{eq::linear_classK_example}, we have
\begin{subequations}
    \eqn{
    \psi_i^0 &= h_i \\
\psi_i^1 &= \dot \psi^0 + \nu_i^1 \psi_i^0 \\
\psi_i^2 &= \ddot \psi^0 + \dot \nu_i^1 \psi_i^0 + \nu_i^1 \dot \psi_i^0 + \nu_i^2( \psi_i^0 + \nu_i^1 \psi_i^0) \\
\psi_i^3 &= \dddot \psi^0 + \ddot \nu_i^1 \psi_i^0 + 2 \dot \nu_i^1 \dot \psi_i^0 + \dot \nu_i^1 \ddot \psi_i^0 + \dot \nu_i^2( \psi_i^0 + \nu_i^1 \psi_i^0) \nonumber \\
& \quad + \dot \nu_i^2( \psi_i^0 + \nu_i^1 \psi_i^0) + \nu_i^3 ( \psi_i^2 ) 
}
\label{eq::derived_barriers_example}
\end{subequations}
and so on. In general, $r_i-k$ derivatives of $\theta_i^k$ are required for computing the terms in \eqref{eq::hobcf_cbf_derivatrive}. For example, for a first-order function $h_i$ with $r_i=1$, no derivatives of $\theta_i^1$ are required. In this work, we do allow $\theta_i^1$ to change even if $r_i=1$. Therefore, in general, we consider $r_i-k+1$ derivatives $\theta_i^k$ and design $(\nu_i^k)^{(r_i-k+1)}$. We denote the state resulting from $\theta_i^k$ dynamics as $\Theta_{\alpha_i^k}$ and their derivatives as $\dot \Theta_{\alpha_i^k}$. For example, for \eqref{eq::derived_barriers_example}, $\Theta_{\alpha_i^k}=[\nu_i^1, \dot \nu_i^1, \ddot \nu_i^1, \nu_i^2, \dot \nu_i^2]^T$ and 
\eqn{
\dot \Theta_{\alpha_i^k} = \begin{bmatrix}
    \dot \nu_i^1 \\
    \ddot \nu_i^1 \\
    \dddot \nu_i^1 \\
    \dot \nu_i^2 \\
    \ddot \nu_i^2
\end{bmatrix} = \begin{bmatrix}
    0 & 1 & 0 & 0 & 0 \\ 
    0 & 0 & 1 & 0 & 0 \\
    0 & 0 & 0 & 1 & 0 \\
    0 & 0 & 0 & 0 & 1 \\
    0 & 0 & 0 & 0 & 0 \\
\end{bmatrix}\Theta_{\alpha_i^k} + \begin{bmatrix}
    0 & 0 \\
    0 & 0 \\
    1 & 0 \\
    0 & 0 \\
    0 & 1
\end{bmatrix}\begin{bmatrix}
    \dddot \nu_i^1 \\
    \ddot \nu_i^2
\end{bmatrix}
\label{eq::Theta_dynamics}
}
As before, we denote with vectors $\Theta$ and $\dot \Theta$, the concatenation of parameters and their derivatives for $\forall i\in \{1,..,N\}, k\in \{1,..,r_i\}$. 

In this work, we develop online adaptation laws for $\dot \Theta_{\alpha}$, namely $\dddot \nu_i^1, \ddot \nu_i^2$ in the example \eqref{eq::Theta_dynamics}. We present the formal objective and problem statement in the next section.

\section{Problem Formulation}
\label{section::problem_formulation}

Given a barrier function $h_i$, $i\in \{1,..,N\}$, and the functions $\psi_i^k, k\in \{1,..,r_i\}$ defined in \eqref{eq::hobcf_barriers}, the control set that satisfies the CBF derivative constraint \eqref{eq::hobcf_cbf_derivatrive} for a given state $x$ at time $t$ is given by
\eqn{
\mathcal{U}_i(t,x) = \{ u \in \reals^m ~|~ \dot \psi_i^{r_i-1}(t,x,u) \geq -\alpha_i^{r_i} (\psi_i^{r_i-1}(t,x)) \}.
}
The intersection of all such sets with the control input domain is given as
\eqn{
    \mathcal{U}_c(t,x) = \bigcap\limits_{i=1}^{N} ~ \mathcal{U}_i(t,x) \bigcap \mathcal{U}.
    \label{eq::cbf_control_intersection}
}
\begin{definition}
    (\textbf{Compatible}) A set of $N$ control barrier functions $h_i(t,x)$, $i\in\{1,\dots,N\}$ is called compatible at time $t$ and at state $x$ if $\mathcal{U}_c(t,x)\neq\emptyset$. 
\end{definition}

We consider the following class of controllers, which is  commonly employed to enforce multiple CBF constraints:
\begin{subequations}
        \begin{align}
     u = \arg \min_{u\in \mathcal{U}} \quad &  ||u-u_r(t,x)|| + M \delta^2 \\
         \textrm{s.t.} \quad & \dot{V}(t,x,u) \leq - k V(t,x) + \delta, \label{eq::qp_hoclf_constraint} \\
          & \dot \psi_i^{r_i-1}(t,x,u) \geq -\alpha_i^{r_i} (\psi_i^{r_i-1}(t,x)), \label{eq::qp_hocbf_constraint} \; \\ & \quad \quad \quad \quad  \quad \quad \quad \quad \quad 
 i\in \{1,2,..,N\} \nonumber
        \end{align}
        \label{eq::HOCBF-CLF-QP}
\end{subequations}
where $u_r: \reals^+\times \reals^n\rightarrow \reals^m$ is the reference control input that is designed without any regard to constraints, $M\in \reals$ is positive definite weighting matrices, $V(t,x)$ a control Lyapunov function (CLF) encoding convergence objectives for the system trajectories, $k$ is the exponential rate of convergence, and $\delta$ is a slack variable used to relax the CLF constraint \eqref{eq::qp_hoclf_constraint}. The optimization \eqref{eq::HOCBF-CLF-QP} is a QP when the dynamics is control-affine as in \eqref{eq::dynamics_general} and $\mathcal{U}$ can be expressed in the form of a polytope $Au\leq b, A\in \reals^{q\times m}, b\in \reals^{q\times 1}, q>0$. 

A solution to \eqref{eq::HOCBF-CLF-QP} exists at a given $(t,x)$ if the set of $h_i(t,x)$, $i\in\{1,\dots,N\}$ is compatible. Note that even when $\mathcal{U}=\reals^m$, that is, when the control input is unbounded, the set of $h_i(t,x)$ may not be compatible. One reason is that depending on the choice of functions $\alpha_i^k$, and the resulting reduction of the size of the safe set,  \eqref{eq::qp_hocbf_constraint} may not admit a solution. 

Therefore, in our approach, we allow adaptation of the parameter vector $\theta_\alpha$ to enforce the compatibility between constraints for all times, and also to tune the response of the controller \eqref{eq::HOCBF-CLF-QP} to meet performance objectives, for example, to make controller less or more conservative in the sense of allowing the trajectories to approach more or less closer to the boundary of the safe set.
We now state our objective.


\begin{objective}
    Consider the dynamical system \eqref{eq::dynamics_general} subject to $N$ barrier constraints $h_i(t,x)\geq 0, i\in\{1,..,N\}$, $\forall (t,x)\in\mathbb R^+\times \mathcal X$, with relative degree $r_i$ w.r.t. \eqref{eq::dynamics_general}. Let the $r_i$ functions $\psi_i^k, k\in\{1,..,r_i\}$ defined in \eqref{eq::hobcf_barriers} with corresponding $\classK$ functions $\alpha_i^k$ with parameters $\theta_{\alpha^k_i}$, giving rise to the HO-CBF condition \eqref{eq::hobcf_cbf_derivatrive}. Let the control input be designed as in \eqref{eq::HOCBF-CLF-QP} and let $\mathscr{C}(t) = \cap_{i=1}^{N} \mathscr{C}_i(t)$. If at $t=0$, $\mathscr{C}(0,x(0))\neq \emptyset$, then design an update law $\dot \Theta: \reals^+ \times \reals^n \times \reals^{n_\alpha}\rightarrow n_{\alpha}$ so that $\mathcal{U}_c(t,x(t))\neq \emptyset ~ \forall t\geq 0$.
\end{objective}

\section{Rate-Tunable CBFs}
\label{section::RT-CBF}
In this section, we introduce our notion of RT-CBFs and analyze the resulting safety guarantees. 

Consider the system dynamics in (\ref{eq::dynamics_general}) augmented with the state $\Theta_\alpha\in\reals^{n_\alpha}$ that obeys the dynamics
\eqn{
\begin{bmatrix} \dot x \\ \dot \Theta_{\alpha_i} \end{bmatrix} = \begin{bmatrix} f(x) + g(x) u \\ f_{\alpha_i}(x, \Theta_{\alpha_i}) \end{bmatrix},
\label{eq::augmented}
}
where $f_{\alpha_i}:\mathcal{X}\times \mathcal A_i \rightarrow \mathcal A_i$ is a locally Lipschitz continuous function w.r.t. $(x,\Theta_{\alpha_i})$, where $\mathcal{A}_i
\subset \reals^{n_{\alpha_i}}$ is a compact set, $i \in \{1,\dots,N\}$.

Assume that $\Theta_{\alpha_i}(t)\in \mathcal{A}_i$, $\forall t\geq 0$, where $\mathcal{A}_i
\subset \reals^{n_{\alpha_i}}$ is a compact set. 
Let the set of allowable states $\s_i(t)$ at time $t$ be defined as the 0-superlevel set of a function $h_i:\mathbb R^+\times\mathcal{X}\rightarrow \reals$ as in \eqref{eq::safeset}. Suppose $h_i$ has relative degree $r_i$ w.r.t the control input $u$ and define functions $\psi^k_i$ as: 
\begin{subequations}
    \eqn{
   \psi^0_i(t,x) &= h_i(t,x), \\
   \psi^k_i(t,x) &= \dot \psi^{k-1}_i(t,x,\Theta_{\alpha_i^k}) + \alpha^k_i (\psi^{k-1}_i(t,x,\Theta_{\alpha_i^k})),  \label{eq::rt_hobcf_barriers} \\
   & \quad \quad \quad \quad \quad k\in \{1,2,..,r_i-1\}.\nonumber
}
\end{subequations}

\begin{definition}(Rate-Tunable CBF) 
A (single) constraint function $h_i:\mathbb R^+ \times \mathcal X\rightarrow \reals$ is a Rate-Tunable Control Barrier Function (RT-CBF) for the set $\mathcal{S}_i(t)$ under the augmented system \eqref{eq::augmented} if for every initial state $x(0)\in \mathcal{S}_i(0)$, there exists  $\Theta_{\alpha_i}(0)\in \mathcal{A}_i$ 
such that
 
\eqn{
\sup_{u\in \mathcal{U}} \left[
  \dot \psi_i^{r_i-1}(t,x,u, \Theta_{\alpha_i}) + \alpha_i^{r_i} (\psi_i^{r_i-1}(t,x,\Theta_{\alpha_i})) 
  \right] & \geq 0, \nonumber \\
   \forall t\geq 0 &  
\label{eq::rtcbf_derivative_condition}
}
\label{definition::RT-CBF}
\end{definition}
Note that for $\dot \Theta_{\alpha_i}\equiv0$ and $\Theta_{\alpha_i}(0)\equiv \Theta_{\alpha_i}$ (a constant independent of $x(0)$), we recover the definition of the classical CBF. In that regard, RT-CBF is a weaker notion of a classical CBF, which allows for tuning the response of the system.

\begin{remark}
While several works employ heuristics to tune the parameters of the CBF condition \eqref{eq::cbf_derivative} so that a solution to the CBF-QP \eqref{eq::HOCBF-CLF-QP} exists for all $t>0$ \cite{zeng2021safety,wang2022ensuring,garg2019control}, most of these are equivalent to treating the parameter $\nu$ of a linear $\classK$ function $\alpha(h) = \nu h$ as an optimization variable. However, a formal analysis encompassing all these heuristics and other possible ways to adapt the $\classK$ function has been lacking so far, and RT-CBFs aim to bridge this gap in theory and application.
\end{remark}

\begin{remark}
\label{remark::rtcbf_Aset}
    Just like CBF condition in Definition \ref{definition::cbf_definition} assumes the existence of a suitable $\classK$ function $\alpha$ and does not provide a way to find it, Definition \ref{definition::RT-CBF} also does not provide a way to compute the set $\mathcal{A}$ for which we can find a RT-CBF. 
\end{remark}

We now show that the existence of a RT-CBF is a necessary and sufficient condition for safety. We first make the following assumption:
\begin{assumption}
    Every safe trajectory $x(t)$ of the system \eqref{eq::dynamics_general} evolves in a fixed compact set $\mathcal{B}\subset \reals^n$. 
\label{assumption::compactness_safe_trajectories}
\end{assumption}
 We consider this a weak assumption as in practice we do not want the states of our system to approach infinity or have a finite escape time. 

\begin{theorem}
\label{theorem::necessity-RT-CBF}
Consider the augmented system~(\ref{eq::augmented}), and the sets of allowable states $\s_i$ defined as 0-superlevel sets of a $C^{r}$ functions $h_i:\mathcal{X}\rightarrow \reals$ as in \eqref{eq::safeset3} for all $i\in \{1,..,N\}$.  
Then under assumption \ref{assumption::compactness_safe_trajectories}, the sets $\s_i$ are rendered forward invariant if and only if there exists control input trajectory $u(t): \reals^+\rightarrow\mathcal{U}$ such that $h_i$ are compatible RT-CBFs for all time $t\geq 0$.

\end{theorem}

\begin{proof}
    The proof for sufficiency is trivial and is shown first. Given that $h_i$ are RT-CBF, we know from \eqref{eq::rtcbf_derivative_condition} that
    \eqn{
    \dot{\psi}_i^{k-1} \geq -\alpha_i^{k-1} (\psi_i), \forall i \in \{1,..,N\}.
    }
    Since the sets $\mathcal{A}_i$ are compact, there exist \classK functions $\gamma_i^k: \reals\rightarrow \reals$ such that $-\alpha^k_i(\psi_i^k(t,s))\geq -\gamma_i^k(\psi_i^k(t,s)), \forall (t,s)\in \reals^+ \times \reals$. Note that the function $\gamma_i^k$ may be the same parametric function as $\alpha_i^k$ but its parameters are not constrained to lie in $\mathcal{A}_i$. Now consider the following comparison system
    \eqn{
        c_i^{k} = \dot c_i^{k-1} &+ \gamma_i^{k}c_i^{k-1}, k\in \{1,..,r_i-1\} \\
        \begin{bmatrix}
            \dot c^0_i \\
             \dot c^1_i \\
            ... \\
            \dot c^{r_i-1}_i
        \end{bmatrix} & \geq \begin{bmatrix}
            -\gamma^{1}_i(c_i^0)  \\
            -\gamma^{2}_i(c_i^1)  \\
            ... \\
            -\gamma^{r_i}_i(c_i^{r_i-1})
        \end{bmatrix}
        \label{eq::comparison_system_rtcbf}
    }
    Since a control input exists such that $\dot c^{r_i-1}_i\geq -\gamma_i^k(\psi_i^k), \forall i \in \{1,..,N\}$, from \cite[Theorem 3]{xiao2021high}, we have that $c_i^{k_i}\geq 0,  \forall k_i\in \{1,..,r_i\}$. A recursive application of comparison lemma \cite[Lemma 3.4]{khalil2002nonlinear} starting from $\psi_i^{k-1}$ gives us that $\psi_i^{k_i}\geq 0,  \forall k_i \in \{1,..,r_i\}$. Therefore, $h_i(t,x)\geq 0 ~ \forall t$, and thus $x(t)\in \mathcal{S}_i \forall i\in \{1,..,N\}$. 
    
    Next, we show the necessity of RT-CBF. Without loss of generality, let $x'(t)$ be any arbitrary safe trajectory of \eqref{eq::dynamics_general}. Given that $x'(t)$ is a safe trajectory, we have $h_i(t,x')\geq 0$ for all $t\geq 0$, $i\in \{1,..,N\}$. Since $h_i$ is $C^{r_i+1}$ function, $\dot h_i$ is a $C^{r_i}$ function as $\dot{h}_i = \frac{\partial h_i}{\partial x}|_{x=x'}\dot x' + \frac{\partial h_i}{\partial t}$. Now consider 
        \eqn{
             \alpha_i^{'1}(r) = - \inf_{t>0, 0 \leq h_i(t,x')\leq r} L_f h_i(t,x') 
             \label{eq::construct_alpha}
        }
Since $x'$ evolves in a compact set under Assumption \ref{assumption::compactness_safe_trajectories}, the set $\{x ~| ~h_i(t,x) \leq r, \forall t>0\}$ is also compact. Hence, the infimum in \eqref{eq::construct_alpha} exists as $h_i$ is a $C^{r_i+1}$ function. Therefore, $\alpha_i^{1'}$ is a non-decreasing $C^{r_i}$ function. 
        There always exist a $C^{r_i} ~ \classK$ function $\hat \alpha_i^1(r)$ that upper bounds $\alpha_i^{'1}(r)$ yielding $\dot h_i \geq -\hat \alpha_i^1(h_i)$.

        Now consider the higher-order derived barrier functions  $\psi_i^k, k\in \{0,..,r_i-1\}$ and consider 
        \eqn{
          \alpha'^{k+1}_i(r) &= - \inf_{t>0, 0\leq h_i(t,x)\leq r} \dot \psi_i^k(t,x',\dot x', \alpha^{'1}_i,..,\alpha^{'1(k)}_i, ... ,\nonumber \\
          & \hspace{4cm} \alpha^{'k}_i, ,..,\alpha^{'k(1)}_i)
        }
        where $\dot \psi_i^{k+1}$ includes all $k+1-j$ time derivatives of functions $\alpha_i^j, j\in \{1,..,k\}$. Then, following similar arguments, we have that $\alpha'^{k+1}_i(r)$ is non-decreasing $C^{r_i-k}$ function. Therefore, there always exists a $C^{r_i-k} ~ \classK$ function $\hat \alpha^1_i(r)$ that upper bounds $\alpha'^{k+1}_i(r)$ such that $\dot \psi_i^k \geq - \hat \alpha^1_i(r)$.
        Since for any given initial state $x'(0)$ and a safe trajectory $x'(t)$, there exists an augmented system \eqref{eq::augmented} such that $h_i$ is a RT-CBF, we conclude that RT-CBF is a necessary condition for safety.

\end{proof}

Note that in contrast to the necessity proof in \cite{ames2016control} for first-order CBFs, we have not assumed that the barriers $h_i$ define a compact set. This is also a result of RT-CBF being a more general notion of safety as CBFs are a subclass of RT-CBF.
The new notion of RT-CBFs is thus more expressive than CBFs with fixed $\classK$ function. It is also trivial to show that even linear $\classK$ functions of the form \eqref{eq::linear_classK_example} with time-varying parameter $\nu$ are equivalent to any fixed nonlinear sufficiently differentiable $\classK$ function. 

The existence of a CBF is shown to be necessary for safety under some conditions\cite{ames2016control}. However, a CBF-QP controller may not be able to result in every possible safe trajectory, that is, it is not complete, and would also require tuning like any other controller. The following theorem illustrates how the tuning of parameter $\Theta$ can be used to shape the response of the CBF-QP controller. The response shaping is essential for increasing either the aggressiveness or conservatives of the system at a given state $x$.

\begin{theorem}
\label{theorem::alpha_minimum_norm}
    Consider the system \eqref{eq::dynamics_general}, a first-order candidate barrier function $h$, a function $\alpha(h,x): \reals\times\reals^n \rightarrow \reals$ and the following CBF-QP controller with unbounded control input
    \begin{subequations}
    \eqn{
        u_{QP} = \min_{u\in \reals^m} \quad &  ||u-u_{r}(t,x)||^2 \\
        \textrm{s.t.} \quad &  \dot h(t,x,u) \geq -\alpha(x,h(t,x))
        }
        \label{eq::cbf_qp_simple}
\end{subequations}
where $u_{r}: \reals^+ \times \reals^n\rightarrow $ is the reference (nominal) control input. Let $u_d(t,x):\reals^+ \times \reals^n\rightarrow \mathcal{U}$ be any desired safe response of the system and $u_r\neq u_d$ w.l.o.g. Then the following choice of function $\alpha(\cdot, \cdot)$ minimizes the norm $||u_d - u_{QP}||$ 
\eqn{
\alpha(x,h) \!= \!\left\{  
      \begin{array}{ll}
         \!\sqrt{2}\frac{L_gh}{||L_gh||}(u_r\! -\! u_d)  &  \textrm{if } ||L_gh||>0\\ 
         \!\quad  - L_fh - L_ghu_r & \textrm{     and } L_gh u_d\!>\!L_gh u_r\\ \\
         \! - L_fh - L_ghu_r & \textrm{if } ||L_gh||=0 \\ \\
          \!\sqrt{2}\frac{L_gh}{||L_gh||}(u_r\! -\! u')  &  \textrm{if } ||L_gh||>0 \\
          \!\quad   - L_fh - L_ghu_r &
          \textrm{     and } L_gh u_d\!<\!L_gh u_r
      \end{array}
\right.   
\label{eq::alpha_minimum_norm}
}
where 
\eqn{
u' = u_d + \frac{1}{\sqrt{2}}L_gh^T L_gh (u_r-u_d).
\label{eq::u_special_case}
}

\end{theorem}
\begin{proof}

    The solution of the optimization \eqref{eq::cbf_qp_simple}, assuming it exists, can be found analytically (see Appendix \ref{appendix::analytical_qp_soln}) to be \eqn{
u_{QP} &= u_r + \nonumber \\
    &\left\{ 
    \begin{array}{ll}
        \frac{1}{\sqrt{2}}\frac{L_gh^T}{||L_gh||}\left( \frac{\partial h}{\partial t}\! +\! L_fh\!+\!L_gh u_r\!+\! \alpha(x,h)\right) & \\
        \quad  \quad  \quad \quad \quad\quad\quad\quad\quad ~\textrm{if } \psi^1 < 0 \textrm{ and } ||L_gh||>0 & \\
        0 \quad \quad  \quad \quad \quad\quad\quad\quad\quad \textrm{if } \psi^1 \geq 0 \textrm{ or }  ||L_gh||=0
    \end{array}
\right.
}
where $\psi^1 = \frac{\partial h}{\partial t}L_f(x)h(t,x)+L_g(x)h(t,x)u_r+\alpha(x,h)$. 

\noindent \textbf{Case 1:} Suppose $\psi^1<0$, $L_gh u_d > L_gh u^r$, and $||L_gh||>0$
\eqn{
u_d - u_{QP} &= u_d - u_r  + \nonumber \\
& \quad \frac{1}{\sqrt{2}} \frac{L_gh^T}{||L_gh||}\left(\frac{\partial h}{\partial t}\!+\!L_fh\!+\!L_gh u_d\!+\! \alpha(x,h)\right)
\label{eq::thm3_diff_equation}
}
Note that $u_{QP}=u_d$ is possible only when $u_d-u_r$ belongs to the span of vector $L_gh^T$. Since $\alpha(x,h)$ only affects the component parallel to $L_gh^T$, multiplying both sides by $\frac{L_gh^T}{||L_gh||}$ gives us
\eqn{
  \frac{L_gh}{||L_gh||}(u_r-u_{QP}) &= \frac{L_gh^T}{||L_gh||}(u_d-u_r) + \nonumber \\
  & \frac{1}{\sqrt{2}}\frac{\partial h}{\partial t}\! +\! L_fh\!+\!L_gh u_r\!+\! \alpha(x,h)
}
The norm $||u_d-u_{QP}||$ is thus minimized when \eqref{eq::alpha_minimum_norm} holds.
\noindent \textbf{Case 2: } $||L_gh||=0$: 
Note that \eqref{eq::alpha_minimum_norm} ensures that a solution to \eqref{eq::cbf_qp_simple} exists whenever $||L_gh||=0$. Moreover, the only possible solution is $u_{QP}=u_r$.

\noindent \textbf{Case 3: }Whenever $L_gh u_r > L_gh u_d$, $u_d$ cannot be a minimizer of \eqref{eq::cbf_qp_simple} irrespective of the choice of function $\alpha$. However, there may exist a vector $u'\in \mathcal{U}$ such that 
\eqn{
||u'-u_d|| < ||u_r-u_d|| \textrm{ and } L_gh u' > L_gh u_r
}
$u'$ can thus be found using the following QP
\begin{subequations}
  \eqn{
\min_{u'} \quad & ||u'-u_d ||^2 \\
\textrm{s.t.} \quad & L_gh u'\geq L_gh u_r
}    
\end{subequations}
The solution of above QP can also be found analytically to be given by \eqref{eq::u_special_case}. Following the same reasoning as in Case 1, we can design $\alpha$ so that $u_{QP}=u'$ and this function minimizes $||u_{QP}-u_d||$.

\end{proof}

The simple result in Theorem \ref{theorem::alpha_minimum_norm} gives us some important insights. First, for different desired responses $u_d$ (such as conservative or aggressive) at state $x$ and time $t$, the function $\alpha$ can be used to steer $u_{QP}$ close to $u_d$. Second, to achieve the aforementioned steering, the function $\alpha$ cannot be just a $\classK$ of the barrier function $h$ as \eqref{eq::alpha_minimum_norm} depends not only on $h$ but also on $x$. In our framework of RT-CBF, the parameter $\Theta$ is a function of $t,x,h$ and thus can fulfill this objective at the points where $\alpha$ in \eqref{eq::alpha_minimum_norm} is differentiable. 
For example, for the linear \classK in \eqref{eq::linear_classK_example}, the following parameter value always satisfies the above equation when $h>0, ||L_gh||>0$
\eqn{
\nu(t,h,x) = \frac{1}{h} \left( \sqrt{2}\frac{L_gh}{||L_gh||}(u_r - u_d) -  \frac{\partial h}{\partial t}+L_fh + L_ghu_d \right)
}
\begin{remark}
    The barrier function in Theorem \ref{theorem::alpha_minimum_norm} could be the constraint function as in RT-CBF or it could be a valid CBF. In either case, tuning parameters of \classK function helps shape response of the controller. Note that \eqref{eq::thm3_diff_equation} suggests that an alternative and definitive way to enforce $u_{QP}=u_d$ would be to change the function $h$ instead of the function $\alpha()$. However, this would require designing $h$ not only to encapsulate safety but also performance objectives. While not impossible, we leave the comparison of these two different ways for future work. In contrast to works that learn barrier function $h$ for a specific task and a predefined function $\alpha$\cite{dawson2022safe,qin2021learning,robey2020learning}, the objective of this work is to explore the advantages of adapting the \classK function as done in \cite{ma2022learning,parwana2022recursive,zeng2021enhancing}.
\end{remark}

\begin{remark}
    In addition to our previous works\cite{parwana2022recursive}\cite{parwana2022trust}, RT-CBF may strike a remarkable resemblance to Parameter Adaptive CBF (PACBF) introduced in \cite{xiao2021adaptive}. The two formulations are closely related and here we point out the differences. PACBF considered a single barrier function while we consider multiple barriers simultaneously. PACBF multiplies a scalar parameter to a generic fixed \classK function whereas we allow for any arbitrary parametric \classK function. PACBF also considers only $r_i-k$ derivatives of the parameters of function $\alpha_i^k$ whereas we consider $r_i-k+1$ derivatives. That is, for a second order barrier function and linear \classK functions in \eqref{eq::linear_classK_example}, PACBF only designs $\dot \nu_i^1$ and, similar to works in \cite{zeng2021enhancing, garg2019control, garg2022control}, allows $\nu_i^2$ to be an optimization variable in QP in the simulation results. Whereas we design $\ddot \nu_i^1$ and $\nu_i^2$. As a consequence, since $\dot \psi_i^1$ is dependent on $\dot \nu_i^1$, PACBF simultaneously optimizes for the control input $u$ and the derivative $\dot \nu_i^1$ in \eqref{eq::RT-HOCBF-CLF-QP}. In contrast, in our case, \eqref{eq::RT-HOCBF-CLF-QP} is independent of $\ddot \nu_i^1$ and $\dot \nu_i^2$ and hence is optimized only for $u$. In Section \ref{section::dynamics_design}, we illustrate that some of the existing methods fall under the framework of RT-CBFs. Not all of them however cannot be expressed as PACBFs. Moreover, in addition to feasibility, we also emphasize the performance of the closed-loop response of the system. For this purpose, it becomes essential to also vary the parameters $\nu_i^{r_i}$ as implied by the Theorem \ref{theorem::alpha_minimum_norm} for $r_i=1$ and also seen in our simulation results in Section \ref{section::simulation}. Finally, the conclusion in PACBF of the existence of a small enough \classK parameter\cite[Theorem 5]{xiao2021high} that leads to feasibility will not necessarily hold for time-varying barrier functions that we consider as the constraints may not allow the dynamical system to be conservatively confined to a small subset of state space.
\end{remark}

In the remainder, we consider the following CBF-QP controller with parametric \classK functions
\begin{subequations}
        \begin{align}
     u = \arg \min_{u\in \mathcal{U}} \quad &  ||u-u_r||^2 + M \delta^2 \\
         \textrm{s.t.} \quad & \dot{V}(t,x,u) \leq - k V(t,x) + \delta, \label{eq::rt_qp_hoclf_constraint} \\
           \dot \psi_i^{r_i-1}(t,x,u, \Theta_{\alpha_i}) &+ \alpha_i^{r_i} (\psi_i^{r_i-1}(t,x,\Theta_{\alpha_i})) \geq 0, \\ 
          & \quad \quad \quad \quad i\in \{1,2,..,N\}\nonumber. \label{eq::qp_rt_hocbf_constraint} \;
        \end{align}
        \label{eq::RT-HOCBF-CLF-QP}
\end{subequations}
While we have established the necessity of RT-CBFs, finding a valid update law $\dot \Theta$ for parameters, much like finding a valid CBF in the sense of \eqref{eq::cbf_derivative}, is non-trivial. The next section is devoted to designing some suboptimal and heuristic methods for ensuring that \eqref{eq::RT-HOCBF-CLF-QP} admits a solution for all time.

\section{Design of $\classK$ function dynamics}
\label{section::dynamics_design}

In this section, we present some algorithms for designing RT-CBFs. Before proceeding further, we illustrate that some of the existing works on adapting \classK function parameters fall under the framework of RT-CBFs. 
\begin{itemize}
    \item The work in \cite{ma2022learning} performs offline learning to model \classK function as a neural network whose output is not only dependent on $h$ but also on the initial state. Their objective is to improve the performance, defined as summation of stage-wise costs over a time horizon. Their framework, like RT-CBFs, have initial parameter value $\Theta(0)$ dependent on the initial state, however $\dot \Theta=0$.
    \item In our previous work \cite{parwana2022recursive}, we perform online optimization of parameters $\Theta$ with a predictor-corrector method for a discrete-time system by posing the system as a differentiable graph whose nodes represent states and inputs. With control input nodes designed as using differentiable QP layers\cite{diamond2016cvxpy,agrawal2019differentiating}, we perform backpropagation is used to compute gradients of future predicted states w.r.t $\Theta$. These gradients are then used to update $\Theta$ to either improve the future performance or feasibility of QPs in case they are infeasible. The initial $\Theta$ is given by the user but $\dot \Theta$ is given by the backpropagation scheme. 
    \item The works in \cite{zeng2021enhancing, garg2019control, garg2022control} pose $\nu$ in the linear \classK function \eqref{eq::linear_classK_example} as an optimization variable in the following CBF-QP
    \begin{subequations}
        \eqn{
    u(t), \nu_i^*(t) \!=\! \min_{u, \nu_i} & ~ ||u-u^r||^2 \!+\! M \sum_{i=1}^N (\nu_i-\nu_{i,r})^2\\
    \textrm{s.t.} &~ \dot h_i(t,x,u) + \delta_i \geq -\nu_i h_i , \nonumber \\
    & \quad \quad \quad \quad \forall i\in \{1,..,N\}
    }
    \label{eq::cbf_nu_qp}
    \end{subequations}
    
    where $\nu_{i,r}$ is a nominal $\nu_i$ value and $M>>1$ is used to promote $\nu_i=\nu_{i,r}$ whenever possible. Note that the aforementioned works only employed first-order CBFs. Assuming a sampled-data system with sampling time $\Delta t$, the above method is equivalent to following the update law
    \eqn{
    \dot \nu_i(t) = \frac{\nu_i^*(t+\Delta t) - \nu_i(t)}{\Delta t},
    }
    where $\nu_i(t)=\nu_i^*(t)$, and $\nu_i(t+\Delta t)$ can be computed by predicting $x(t+\Delta t)$ whenever the dynamics is known.
\end{itemize}

In the subsequent, inspired by aforementioned works, we devise a pointwise sufficient condition for higher-order barrier functions and provide a lower bound on $\dot\Theta_{\alpha_i}:\reals^n \times \reals^{n_\alpha}\rightarrow \reals^{n_\alpha}$ for enforcing persistent feasibility. Finally, we present a case study on decentralized multi-robot control that showcases the potential of RT-CBFs in reducing conservativeness in the responses of the agents while enforcing safety requirements. 

\subsection{Strategy for designing PointWise sufficient condition}
\label{section::suboptimal_parameter_dynamics}
Algorithm \ref{algo::rt_cbf_point_optimal} presents an update law that minimally modifies $\Theta$ to ensure the feasibility of the CBF-QP controller. 

\textbf{Case 1 } $k\in \{1,..,r_i\}$: 
At a given time $t$ and state $x(t)$, the following equations much be satisfied
\eqn{
&\dot \psi_i^{k-1}(t,x) + \nu_i^k \psi_i^{k-1}(t,x) \geq 0 \\
&\implies \nu_i^k \geq -\frac{\dot\psi_i^{k-1}(t,x)}{\psi_i^{k-1}(t,x)}
\label{eq::nu_trajectory}
}

 Since we differentiate $\nu_i^k$ $r_i-k+1$ times in our framework, we design $(\nu_i^{k})^{r_i-k+1}$ to ensure the constraint  \eqref{eq::nu_trajectory} is satisfied at all times. Note that this could be done, for example, by posing $\nu_i^k$ as a CBF as done in \cite{xiao2021adaptive}, but our framework is independent of the exact method chosen. Below, we present a simple method for sampled data systems that avoids designing a new CBF function for $\nu_i^k$. 

\subsubsection*{Application in Sampled-Data systems}

For a sampling time $\Delta t$, we predict the future state $x$ at time $t+\Delta t$ using dynamics \eqref{eq::dynamics_general} and control input from \eqref{eq::RT-HOCBF-CLF-QP}. Given a desired $\nu_{i,d}^k(t+\Delta t)\geq0$ by the user, $\nu_i^k(t+\Delta t)$ is chosen as
\eqn{
  \nu_i^k(t\!+\!\Delta t) \!=\! \max \left(  -k\frac{\dot \psi_i^{k-1}(t\!+\!\Delta t)}{\psi_i^{k-1}(t\!+\!\Delta t, x(t\!+\!\Delta t))}, \nu_{i,d}^k(t\!+\!\Delta t)  \right)
  \label{eq::sampled_data_rule}
}
where $k\geq 1$ is a tuning parameter to avoid satisfying the constraint \eqref{eq::nu_trajectory} at the boundary. Choosing $k=1$ may lead to a very conservative set of allowed initial states for which the $h_i$ are found to be persistently compatible. 

Next, since the $\nu_i^k$ dynamics is given a simple integrator model, its analytical solution is known to be
\eqn{
\nu_i^k(t+\Delta t) = & \nu_i^k(t) + (\nu_i^k(t))^{(1)} * \Delta t + (\nu_i^k(t))^{(2)} \frac{\Delta t^2}{2} + ... \nonumber \\
& + (\nu_i^k(t))^{(r_i-k+1)}\frac{\Delta t^{(r_i-k+1)}}{(r_i-k+1)!}
}
This motivates the following design of $(\nu_i^k)^{r_i-k+1}$ 
\eqn{
  (\nu_i^k(t))^{(r_i-k+1)} & = \frac{(r_i-k+1)!}{\Delta t^{(r_i-k+1)}} \Biggl( \nu_i^k(t+\Delta t) - \nu_i^k(t) -  \nonumber \\
  & \quad \quad \left. \sum_{p=1}^{r_i-k}(\nu_i^k(t))^{(p)} * \frac{\Delta t^p}{p!}\right)
  \label{eq::nu_final_derivative}
}

\textbf{Case 2 $k=r_i$}: $\nu_i^{r_i}$ appears in \eqref{eq::qp_rt_hocbf_constraint} and is thus dependent on the chosen control input. However, when there is a single barrier constraint, a lower bound can still be obtained by considering the worst-case scenario
\eqn{
\nu_i^k \geq -\frac{\sup_{u\in \mathcal{U}} \dot \psi_i^{r_i-1}(t,x,u,\Theta)}{\psi_i^{r_i-1}(t,x,\Theta)}
\label{eq::nu_final_lower_bound}
}
In case of multiple constraints, assuming all constraints are equally important, the lower bounds of $\nu_i^k$ are solutions to the following linear program
\begin{subequations}
  \eqn{
    \min_{\nu_i^k, u\in \mathcal{U}} \quad & \sum_{i=1}^{N} \nu_i^k\\
    \textrm{s.t.} \quad & \dot \psi_i^{r_i-1} \geq -\nu_i^{r_i}\psi_i^{r_i-1} \\
    & \nu_i^k \geq 0, i\in \{1,..,N\}
}    
\label{eq::nu_final_qp}
\end{subequations}

With the bound in \eqref{eq::nu_final_lower_bound}, the design of $\dot\nu_i^{r_i}$ is the same as in Case 1. Note that for sample-data systems, the optimization \eqref{eq::nu_final_qp} is solved for the states the time $t+\Delta t$ to first find desired $\nu_i^{r_i}$

\begin{example}
    Consider a second-order barrier function $h$ with input bounds $\mathcal{U}$ imposed on the system \eqref{eq::dynamics_general}. The first two derived barrier functions are given by
    \begin{subequations}
        \eqn{
    \psi^0 &= h \label{eq::example_h0}\\
    \psi^1 &= \dot h + \nu^1 h \label{eq::example_h1}\\
    \psi^2 &= \dot \psi^1 + \nu^2 \psi^1 = \ddot h + \nu^1 \dot h + \dot \nu^1 h + \nu^2 (\dot h + \nu^1 h) \label{eq::example_h2}
    }
    \end{subequations}
    
    From \eqref{eq::example_h1}, we have following bound on $\nu^1$
    \eqn{
    \nu^1 \geq -\frac{\dot h}{h}
    }
    Finally, the constraint on $\nu^2$ is obtained as follows
    \eqn{
    \nu^2 \geq -\frac{\sup_{u\in \mathcal{U}}\ddot h(t,x,u)-\nu^1\dot h - \dot \nu^1 h}{\dot h + \nu^1 h}
    }
    
\end{example}

\begin{algorithm}[h!]
\caption{RT-CBF Pointwise Optimal Controller}
\begin{algorithmic}[1]
\Require $t, x(t), N$ \Comment{current time, current state, No. of constraints}
    \For {i=1 to N} \Comment{Compute all derived barriers}
    \State Compute $\dot \psi_i^{r_i-1}, \psi_i^{r_i}$
    \EndFor
    \State Compute $u(t)$ using \eqref{eq::RT-HOCBF-CLF-QP}\Comment{Implement $u$ on system \eqref{eq::dynamics_general}}
    \State Compute $x(t+\Delta t)$ using $u(t)$ and \eqref{eq::dynamics_general}
    \For {i = 1 to N} \Comment{Update $\classK$ parameters}
        \For {$k=1$ to $r_i-1$} \Comment{for all derived barriers}
            \State Compute $\nu_{i}^{k}$ and its derivatives using \eqref{eq::nu_trajectory},\eqref{eq::nu_final_derivative},\eqref{eq::nu_final_lower_bound}
        \EndFor
        
    \EndFor
    \State Compute $\Theta (t+\Delta t)$ \Comment{Update parameters with designed update law}
\end{algorithmic}
\label{algo::rt_cbf_point_optimal}
\end{algorithm}

\section{A Synthesis Case Study: Designing RT-CBFs For Multi-Robot Systems}
\label{section::multi-agent}
In this section, as an application case study, we utilize RT-CBFs and propose $\classK$ function parameter dynamics for decentralized multi-robot control based on our work in \cite{parwana2022trust}. We show through simulations that our online parameter adaptation scheme leads to a less conservative response by quantifying the ease of satisfaction of CBF constraints. 

We consider a multi-robot system with $N$ agents. The state of each agent $l\in \{1,..,N\}$ is denoted as $x_l \in \reals^{n}$ and follows the dynamics 
\eqn{
  \dot x_l = f_l(x_l) + g_l(x_l)u_l 
\label{eq::dynamics_multi_agent}
}
where $f_l:\reals^{n_l}\rightarrow \reals^{n_l}, g_l:\reals^{n_l}\rightarrow \reals^{m_l}$ are locally Lipschitz continuous functions and $u_l \in \mathcal U_l \subset \reals^{m_l}$ is the control input of each agent $l$. Note that in general the agents can be of heterogeneous dynamics, i.e., $f_p\neq f_q$, $g_p\neq g_q$ for $p\neq q$. 

The set of neighbors of agent $l$ is denoted as $\mathcal{N}_l=\{1,..,N\}\setminus \{l\}$. We consider a decentralized scenario where each agent decides its own control input under the following assumption:
\begin{assumption}
\label{assumption:multi_agent}
    Each agent $l$ has perfect observations of the state $x_o$ and its derivative $\dot x_o$ of every other agent $o\in \mathcal N_l$. 
\end{assumption}

We consider pairwise safety constraints (i.e., constraints affecting a pair of two agents $i,j$) denoted $h_{ij}:\mathbb R^{n_i}\times \mathbb R^{n_j} \rightarrow \mathbb R$, so that their zero-superlevel sets denoted $\mathcal S_{ij}=\{x_i, x_j \; | \; h_{ij}(x_i,x_j)\geq 0\}$ define sets of safe states such that $\partial \mathcal S_{ij}=\{x_i, x_j \;|\; h_{ij}(x_i,x_j)=0\}$ and $\textrm{Int}(\mathcal S_{ij})=\{x_i, x_j \;|\; h_{ij}(x_i,x_j)>0\}$, and their time derivatives along the system dynamics \eqref{eq::dynamics_multi_agent} read:  
\begin{align}
\dot h_{ij}(x_i,x_j)=\frac{\partial h_{ij}}{\partial x_i} \; \dot x_i+\frac{\partial h_{ij}}{\partial x_j} \; \dot x_j.
\label{eq::pairwise-derivative}
\end{align}
We note that for $(x_i, x_j)\in \partial \mathcal S_{ij}$, i.e., when $h_{ij}(x_i,x_j)=0$, the sign of \eqref{eq::pairwise-derivative} should be non-negative for the value of the function $h_{ij}$ to not decrease further along the trajectories \eqref{eq::dynamics_multi_agent}, i.e., for the trajectories to remain within $\mathcal S_{ij}$. Based on this, we say that at a given time and states $(t,x_i,x_j)$, an agent $j$ has one of the following behaviors towards agent $i$:
\begin{itemize}
    \item Cooperative: When the dynamics of agent $j$ is such that $\frac{\partial h_{ij}}{\partial x_j}\dot x_j \geq 0$ 
    \item Noncooperative: When the dynamics of agent $j$ is such that $\frac{\partial h_{ij}}{\partial x_j}\dot x_j < 0$.
\end{itemize}


We clarify that the definitions above do not aim to classify intentions (e.g., malicious vs non-malicious) of the agent $j$ w.r.t. agent $i$. For example, in the presence of multiple agents, a possibly non-malicious agent $j$ may still be perceived as noncooperative by agent $i$ if agent $j$ is moving towards agent $i$ under the influence of another agent $k$ that is moving towards $j$. Our objective is not to classify whether agent $j$ is malicious or not, but rather to tune the controller of the agent $i$ so that it responds safely and accordingly to the different types of neighbor agents, and the type of motion exhibited by each one of them. For example, a fast (respectively, slow) adversarial agent $j$ will ideally warrant more (respectively, less) conservative response from the agent $i$. In the sequel we often refer to the agent $i$ as an ego agent, and a neighbor agent $j$ as a neighbor agent.

A dynamic environment with moving agents implies that the world is inherently time-varying from the perspective of an agent $i$. We propose an online adaptation scheme to shape the response of the ego agent $i$ so that it becomes less or more conservative w.r.t. its neighbor agents $j$ based on observations of their states and their derivatives. 
To this end we develop a \textit{trust metric} based on instantaneous observations that is used to modify $\theta_{\alpha_ij}^k$.

\subsubsection{Design of Trust Metric}
\label{section::trust_metric}
Consider the following CBF constraint

\eqn{
    \dot \psi^{k-1}_{ij} & \geq -\alpha_{ij}^{k}( \psi^{k-1}_{ij} ) \nonumber\\
    \implies \frac{\partial \psi^{k-1}_{ij}}{\partial x_i}\dot{x}_i + \frac{\partial \psi^{k-1}_{ij}}{\partial x_j}\dot x_j  + \frac{\psi_{ij}^{k-1}}{\partial \Theta_{ij}}\dot \Theta_{ij}& \geq -\alpha_{ij}^{k-1}( \psi^k_{ij} ) \label{eq::cbf_two_states}
}

\begin{figure}[htp]
    \centering
    \includegraphics[scale=0.25]{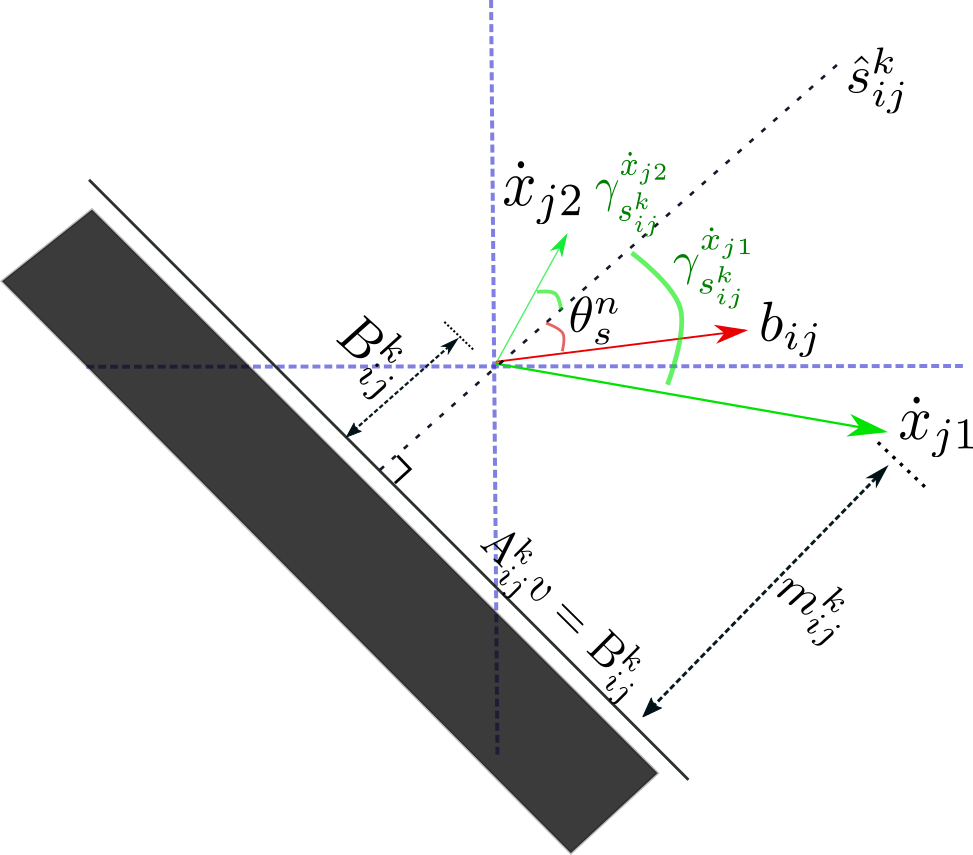}
    \caption{\small{Illustration of the halfspaces induced by the constraint function. $Av\geq b$ represents the halfspace described by \eqref{eq::allowed_halfspace} with $A = \frac{\partial h_{ij}}{\partial x_j}$, $v=\dot{x}_j$, and $b =  -\alpha h - \max \left\{ \frac{\partial h_{ij}}{\partial x_i}\dot{x}_i \right\}$. The two $x_j$ are possible instances of the actual motion of agent $j$. $\hat{s}$ is the normal to the hyperplane $Av=b$. $\hat n_j$ is the nominal direction of motion for agent $j$.}
    }
    \label{fig:allowed_space}
\end{figure}

 where note that although $\dot \Theta_{ij}$ appears in \eqref{eq::cbf_two_states}, the $r_{ij}-k+1$-order derivative to be designed will not contribute in the equation (see \eqref{eq::derived_barriers_example} for example). The margin by which the CBF condition (\ref{eq::cbf_two_states}) is satisfied, i.e., the value of $\dot \psi^{k-1}_{ij} +\alpha_{ij}^{k-1}( \psi^{k-1}_{ij} )$ represents of the ease of satisfaction of constraint w.r.t agent $i$. The best case action by agent $i$ thus corresponds to the largest possible margin for a given motion $\dot{x}_j$ of agent $j$. 
 
\noindent \textbf{Constraint margin based trust score:} The motions of agent $j$ that lead to a feasible solution to (\ref{eq::cbf_two_states}) are considered as \textit{safe} by $i$. Consider the following decomposition
   \eqn{
      \dot \psi^{k-1}_{ij} = \frac{\partial \psi^{k-1}_{ij}}{\partial x_i}\dot{x}_i + \frac{\partial \psi^{k-1}_{ij}}{\partial x_j}\dot x_j + \frac{\partial \psi^{k-1}_{ij}}{\partial \alpha^{k}_{ij}}\dot \alpha^{k}_{ij}
   }
   Here, we have ignored the higher order derivatives of $\alpha^{m}_{ij}, m<k-1$. The safe motions of $j$ are thus given by
    \eqn{
     \frac{\partial \psi^{k-1}_{ij}}{\partial x_j}\dot x_j & \geq - \alpha_{ij}^{k}( \psi^{k-1}_{ij} ) -  \frac{\partial \psi^{k-1}_{ij}}{\partial \alpha^{k}_{ij}}\dot \alpha^{k}_{ij} - \frac{\partial \psi^{k-1}_{ij}}{\partial x_i}\dot{x}_i  \nonumber \\
     & \geq - \alpha_{ij}^{k}( \psi^{k-1}_{ij} ) - \frac{\partial \psi^{k-1}_{ij}}{\partial \alpha^{k}_{ij}}\dot \alpha^{k}_{ij} \nonumber \\
     & \hspace{2cm} - \max_{u_i \in \mathcal{U}} \left\{ \frac{\partial \psi^{k-1}_{ij}}{\partial x_i}\dot{x}_i \right\}
\label{eq::allowed_halfspace}
    }
where the last term in \eqref{eq::allowed_halfspace} can be obtained from the following Linear Program (LP):
\begin{subequations}\label{eq::LP}
    \begin{align}      \begin{split}\label{eq::LPa}
             \max_{u_i} \quad & \frac{\partial \psi^k_{ij}}{\partial x_i}\dot{x}_i 
             \vspace{-0.1cm}
        \end{split}\\ 
        \begin{split} \label{eq::LPb}
            \textrm{s.t.} \quad & \dot{\psi}^k_{il} \geq -\alpha_{il} \psi^k_{il}, ~~\forall l\in N_i\setminus j.
        \end{split}
    \end{align}
\end{subequations}
    
The inequality \eqref{eq::allowed_halfspace} gives a lower bound on motions of agent $j$, a violation of which will lead to the infeasibility of QP controller of agent $i$. The design of the trust factor is thus based on a comparison of the contribution of agent $j$ with the aforementioned lower bound. 


The evaluation of  \eqref{eq::allowed_halfspace} as equality represents a halfspace whose separating hyperplane has the normal direction $\hat s_{ij}^k=\frac{\partial \psi^k_{ij}}{\partial x_j}$, and has been visualized in Fig. \ref{fig:allowed_space}.
Let the half-space in \eqref{eq::allowed_halfspace} be represented in the form $A_{ij}^kv \geq B_{ij}^k$ with $v = \dot{x}_j$ and $A_{ij}^k,B_{ij}^k$ defined accordingly. Let $m_{ij}^k$ be the constraint margin by which \eqref{eq::allowed_halfspace} is satisfied
\eqn{
m_{ij}^k = A_{ij}^kv + B_{ij}^k 
}
where $m_{ij}^k<0$ is an incompatible vector, corresponding to scenarios that should never happen. The larger the value of $m_{ij}^k$, the greater the ease of satisfaction for agent $i$. Hence, a larger $m_{ij}^k$ allows room for relaxing the constraint. Therefore, the constraint margin-based trust score is designed as
\eqn{
\rho_{m_{ij}^k} = f_d(m_{ij}^k), k\in \{ 1,..,r \}
\label{eq::rho_d}
}
where $f_d:\reals^+\rightarrow[0,1]$ is a monotonically-increasing, Lipschitz continuous function, such that $f_d(0)=0$. A possible choice would be $f_d(d) = \tanh(\beta d)$, with $\beta$ being a scaling parameter. \\

\noindent \textbf{Belief based trust score:} 
We now describe a method to incorporate $i$'s prior belief of the motion of agent $j$. Note that infinitely many values of $\dot x_j$ may satisfy the CBF derivative condition \eqref{eq::allowed_halfspace} with the same constraint margin (all $\dot x_j$ lying on the yellow hyperplane in Fig. \ref{fig:allowed_space} will have the same $m_{ij}^k$). This gives us an additional degree of freedom for deciding which motions are more trustworthy. Consider the scenario shown in Fig.\ref{fig::infinite_solutions} where an agent $j_1$ is moving directly towards $i$ with speed $2$ and another agent moving at an angle $\phi$ w.r.t the vector $x_i - x_j$ but with speed $2/\cos45 \degree$. Both $j_1$ and $j_2$ would thus lead to the same constraint margin but we would like to trust $j_2$ more if we had a belief that it is supposed to be moving in that direction or because it is not moving straight towards $i$. We denote this nominal belief direction of agent $j$ w.r.t agent $i$ as $b_{ij}$. If $\dot x_j$ performs worse than the nominal direction, then it is a cause of concern as it is not consistent with the $i's$ belief for agent $j$. Note that $b_{ij}$ need not be the true direction of motion of agent $j$ for our algorithm to work since if $b_{ij}$ is not known, we can consider the worst-case scenario where $b_{ij}=x_i-x_j$, i.e., that $j$ is an adversary and intends to collide with agent $i$; over time, our algorithm will learn that $j$ is not moving along $b_{ij}$ and therefore increases $i's$ trust of $j$. Knowledge of nominal direction plays an important role in shaping belief as it helps to distinguish between uncooperative and adversarial agents.

\begin{figure}[htp]
    \centering
    \includegraphics[scale=0.15]{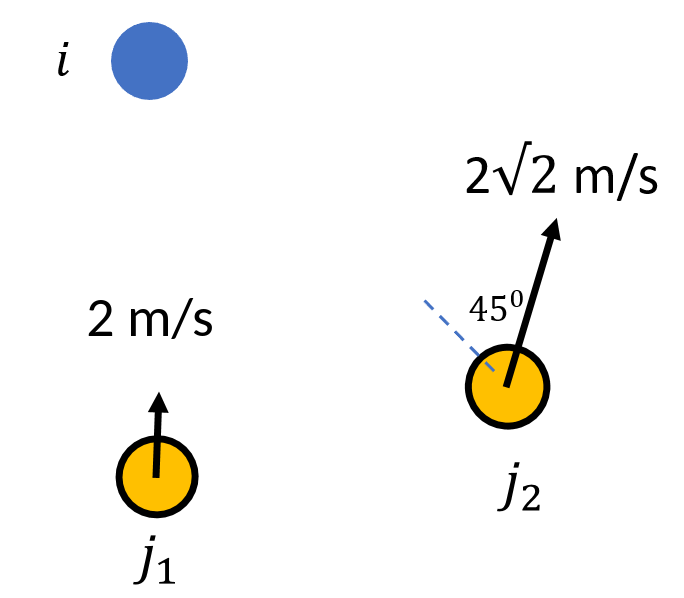}
    \caption{\small{Agents $j_1$ and $j_2$ contribute the same constraint margin for satisfying the CBF derivative condition \eqref{eq::cbf_two_states}.     }
    }
    \label{fig::infinite_solutions}
\end{figure}

Suppose the angle between the vectors $b_{ij}$ and $\hat s_{ij}^k$ is denoted by $\gamma_{s_{ij}}^{b_{ij}}$, and the angle between $\dot x_j$ and $\hat s_{ij}$ is denoted by $\gamma_{s_{ij}^k}^{\dot x_{j}}$. The belief-based trust is designed as
\eqn{
\rho_{\gamma_{ij}} = f_\gamma( \gamma^{b_{ij}}_{s_{ij}^k}/\gamma^{\dot x_j}_{s_{ij}^k} ),
\label{eq::rho_theta}
}
where $f_\gamma:\reals^+ \rightarrow [0,1]$ is again a monotonically-increasing Lipschitz continuous function with $f_\gamma(0)=0$. Note that even if $\gamma^{b_{ij}}_{s_{ij}^k}=\gamma^{b_{ij}}_{s_{ij}^k}$, i.e., $j$ is perfectly following its nominal direction, we do not design $f_\gamma$ to be 1 as the robot might be uncooperative. We give $f_\gamma$ higher values when $\gamma^{b_{ij}}_{s_{ij}^k}<\gamma^{b_{ij}}_{s_{ij}^k}$, as with $a_{j_2}$ in Fig. \ref{fig:allowed_space}, $j$ seems to be compromising its nominal movement direction for improved safety, thus leading to a higher score. Finally, when ${b_{ij}}_{s_{ij}}<\gamma^{b_{ij}}_{s_{ij}}$, as with $a_{j_1}$ in Fig. \ref{fig:allowed_space}, then $j$'s current motion is performing worse for inter-robot safety than its nominal motion and therefore has lower trust than the latter case. \\



\noindent \textbf{Final Trust Score:} The trust metric is now designed based on $\rho_{m_{ij}^k}$ and $\rho_{\gamma_{ij}^k}$. Let $\bar \rho_{m_{ij}^k}\in (0,1)$ be the desired minimum robustness/margin in satisfying the CBF condition \eqref{eq::allowed_halfspace}. Then, the trust metric $\rho_{ij}\in [-1,1]$ is designed as follows:
\eqn{
\rho_{ij}^k = \left\{ \begin{array}{cc}
    (\rho_{m_{ij}^k}-\bar \rho_{m_{ij}^k})\rho_{\gamma_{ij}^k},  & \mbox{if } \rho_{m_{ij}^k} \geq \bar \rho_{m},\\
     (\rho_{m_{ij}^k}-\bar \rho_m)(1-\rho_{\gamma_{ij}^k}),  & \mbox{if } \rho_{m_{ij}^k} < \bar \rho_m.
\end{array} \right.
\label{eq::trust}
}
Here, $\rho_{ij}^k=1$ and $-1$ represent the absolute belief in another agent being cooperative and noncooperative, respectively. If $\rho_{m_{ij}^k}>\bar \rho_m$, then we would like to have $\rho_{ij}^k>0$, and its magnitude is scaled by $\rho_{\gamma{ij}^k}$ with smaller values of $\rho_{\theta_{ij}^k}$ conveying low trust. Whereas, if $\rho_{m_{ij}^k}<\bar \rho_m$, then we would like the trust factor to be negative. A smaller $\rho_{\gamma_{ij}^k}$, in this case, implies more distrust and should make the magnitude larger, hence the term $1-\rho_{\gamma_{ij}^k}$.
The trust $\rho$ is now used to adapt $\alpha$ with following equation
\eqn{
    \dot{\alpha}_{ij}^k = f_{\alpha}(\rho_{ij}^k),
    \label{eq::alpha_dot}
}
where $f_{\alpha}:[-1,1]\rightarrow \reals$
 is a monotonically increasing function. A positive value of $f_{\alpha_{ij}^k}$ relaxes the CBF condition by increasing $\alpha_{ij}^k$, and a negative value decreases $\alpha_{ij}^k$.
Note that \eqref{eq::allowed_halfspace} needs $\dot \alpha_{ij}^{k-1}$. Therefore, we first design $\dot \alpha_{ij}^1$ and then use it in the design of $\dot \alpha_{ij}^2$ and repeat this procedure for higher order derivatives. This is similar to the cascaded control structures.



\subsubsection{Projection to Feasible Parameter Space}
To ensure the feasibility of the QP, the heuristically designed parameter update rules are passed as the desired derivatives in Algorithm \ref{algo::rt_cbf_point_optimal} of Section \ref{section::trust_metric}. For real applications, we employ the sampled-data approach in our simulations. However, since this approach requires projecting to one time step ahead future state, it requires knowledge of the closed-loop dynamics of the multi-agent system which is unknown to individual agents. We therefore implement a robust version of our Algorithm \ref{algo::rt_cbf_point_optimal} based on the following assumption
\begin{assumption}
    The closed-loop of every agent has a bounded change every sampling period, that is, $||\dot x_i(t) - \dot x_j(t-\Delta t)||\leq \Delta F $. This bound is known to the other agent.
    \label{assumption::multi_agent_assumption}
\end{assumption}
Assumption \ref{assumption::multi_agent_assumption} would imply, for example, that if agent $i$ observes the other agent to be moving at $2m/s$ at time $t-\Delta t$, then the actual velocity of $j$ at time $t$ would be $2\pm \Delta v$, where $\Delta v$ is a known constant. This assumption leads to less conservative results than assuming a worst-case bound on dynamics and does not require continuity of controllers. This bound is used in Lines 5-11 of Algorithm \ref{algo::rt_cbf_point_optimal} to update the parameters of \classK function. 

\section{Simulation Results}
\label{section::simulation}
Code and videos available at: \href{https://github.com/hardikparwana/RateTunableCBFs}{https://github.com/hardikparwana/RateTunableCBFs}.
In this section, we present three simulation studies that show the efficacy of our proposed algorithms.

\subsection{Pointwise Optimal Controller}
Consider the green ego agent surrounded by two black non-ego agents shown in Fig. \ref{fig::si_trajs} that is reminiscent of a commonly encountered traffic scenario. The motion of non-ego agents, given by two black lines, is known to the ego agent and the free space available to the ego agent shrinks initially and then becomes constant.
We compare CBF-QP with RT-CBF-QP that implements Algorithm \ref{algo::rt_cbf_point_optimal}. The barrier function used is $h_i=(x-x_i)^2$ where $x$ is the position of the ego agent and $x_i$ is the position of the non-ego agent $i$. We simulate two dynamics for ego agents: 1) single integrator (SI) with first-order barrier functions $h_1,h_2$ and 2) double integrator (DI) with second-order barrier functions $h_1, h_2$. The reference controller is chosen as $u_r(x)=-k_x (x-6.0)$ for SI and $u_r(x) = k_v(\dot x + k_x*(x-6.0))$ for DI where $k_x=1.0, k_v=2.0$ are controller gains. In all our simulations, the fixed parameter value in CBF-QP is also the initial value of parameters for RT-CBF-QP. Fig \ref{fig::si_trajs} shows some of the trajectories that result from the choice of different parameters. The trajectories end at the point of infeasibility or when simulation ends at $t=5.5 s$. The sensitivity of the initial location is further analyzed in Fig. \ref{fig::si_di_state_space} that shows the variation in time to failure for different initial positions but the same \classK parameters. For DI, the initial velocity is zero. Failure is said to happen when either RT-CBF-QP becomes infeasible or the barrier function $h_i\rightarrow 0$ (in which case $\nu_i^k\rightarrow \infty$). The domain of state space that defines a valid initial initial conditions is larger for RT-CBFs compared to CBF.

\begin{figure}[h!]
\centering
         \includegraphics[scale=0.6]{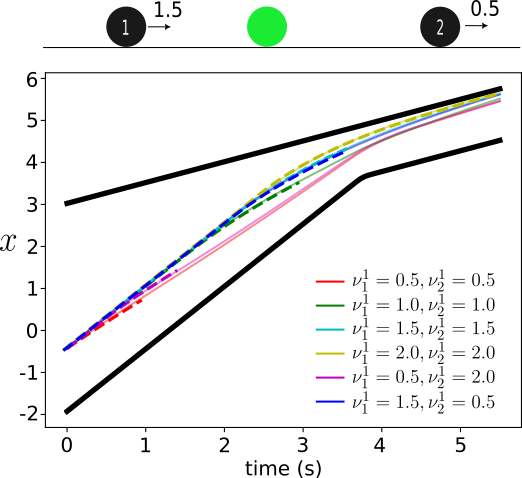}
         \caption{\small{Trajectories with different initial parameter values in CBF-QP \eqref{eq::CBF-CLF-QP} (dotted lines) and RT-CBF-QP \eqref{eq::RT-HOCBF-CLF-QP} (solid lines) for initial $x=-0.5$. Dotted lines end when CBF-QP becomes infeasible. RT-CBF-QP never becomes infeasible and solid lines extend till the end of the simulation. Simulation is run for 5.5 seconds.}}
         \label{fig::si_trajs}
\end{figure}

\begin{figure}[h!]
\centering
         \includegraphics[scale=0.6]{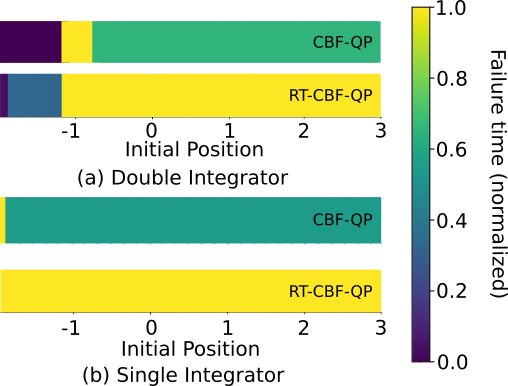}
         \caption{\small{Time of failure as a function of the initial location $x_0\in [-1.95, 3.0]$ of the double integrator agent. The initial velocity is zero. Algorithm \ref{algo::rt_cbf_point_optimal} always increases the time horizon over which the controller is feasible. Time $t$ is normalized as $t/t_f$ with scaling factor $t_f = 5.5$ seconds.}}
         \label{fig::si_di_state_space}
\end{figure}

\subsection{Adaptive Cruise Control}
We simulate the Adaptive Cruise Control (ACC) problem with a decelerating leader. Consider an ego agent moving with velocity $v$ and a leader agent moving at velocity $v_L$ with distance $D$ between them. The dynamics of are given by
\eqn{
\begin{bmatrix}
    \dot v \\
    \dot v_L \\ 
    \dot D
\end{bmatrix} = \begin{bmatrix}
    \frac{-1}{M}F_r(v) + \frac{1}{M} u \\
    a_L \\
    v_L - v
\end{bmatrix}
}
where $M$ denotes the mass of ego agent, $F_r(v)$ is the resistance force defined as $F_r(v)=f_0\textrm{sign} (v) + f_1 v + f_2 v^2$ 
with $f_0,f_1,f_2>0$, and $a_L$ is the acceleration of the leader. The safety objective of the ego agent is to maintain a minimum distance $\bar D$ with the leader. The desired velocity is given by $v_d$. Additionally, the velocity and control input are constrained by
\eqn{
  v_{min} \leq v \leq v_{max} \\
  -c M g \leq u \leq c M g
}
where $c$ is acceleration coefficient and $g$ is acceleration due to gravity. To ensure safety, we formulate three barrier functions as follows
\eqn{
h_1 = D - \bar D, h_2 = v - v_{min}, h_3 = v_{max}-v
}
where $h_1$ is second-order barrier whereas $h_2,h_3$ are first-order. The CLF is chosen as $V = (v-v_d)^2$. We further choose linear \classK functions as in \eqref{eq::linear_classK_example} and apply Algorithm \ref{algo::rt_cbf_point_optimal} to ensure safety. The following parameters are used for the simulation: $M=1650$ kg, $f_0=0.1, f_1=5, f_2=0.25$, $g=9.81m/s^2$, $v_{mijn}=0$, $v_{max}=0$, $v_d=24m/s$, $c=0.4$, $\bar D=10$. The leader is designed to slow down in the first $10s$ as
\eqn{
a_L = \left\{ \begin{array}{cc}
   -0.6(1-\tanh\left( \frac{1}{2(10-t)} \right)  & t\leq 10 \\
    0 & t>10
\end{array}
\right.
}
The objective of the CBF-QP to ACC is designed in the same way as \cite{ames2016control}. We compare CBF-QP with RT-CBF-QP in Figs.\ref{fig::acc_vel}-\ref{fig::acc_control}. The initial value of the parameters are $\nu_i^1=0.5, \nu_1^2=0.7, \nu_2^1=0.5, \nu_2^2=0.5$. The exponential rate of convergence for CLF is $k=10$. The initial state is $v=20m/s, v_L=10m/s, D=100m$. The simulation is run for $50$ seconds. CBF-QP (with fixed parameters) becomes infeasible before the simulation finishes as shown in Fig. \ref{fig::acc_vel}. RT-CBF-QP on the other hand can maintain feasibility and safety at all times. The variations of barrier function, $\nu_1^2$, and $u$ with time are shown in Figs.\ref{fig::acc_barrier}, \ref{fig::acc_param}, \ref{fig::acc_control} respectively. Note that $\nu_1^2$ starts increasing as control input is approached. We design the desired parameter value $\nu_{i,d}^1$ in \eqref{eq::sampled_data_rule} to be a P controller with gain $1$ and saturated at $1$ to steer $\nu_i^2$ to $0.7$. As seen in Fig. \ref{fig::acc_param}, $\nu_i^2$ returns to its nominal value of $0.7$ eventually. Other parameters do not change in this example and their variation is thus not shown. This example illustrates that, for a chosen barrier function $h$, the feasibility of CBF-QP controllers is highly dependent on parameters but online adaptation can help circumvent this issue.

\begin{figure}[h!]
\centering
         \includegraphics[scale=0.5]{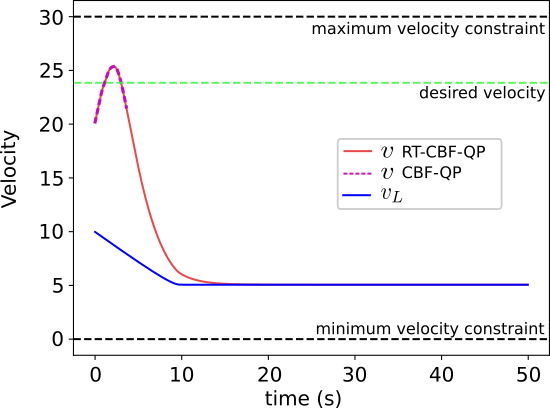}
         \caption{\small{Velocity variation for the ACC problem. The dotted pink line ends when CBF-QP becomes infeasible. The leader decelerates for the first 10 seconds and then moves at constant velocity. The ego-agent approaches the boundary of the safe set (see Fig. \ref{fig::acc_barrier}) and hence does not have enough bandwidth to steer back towards its desired velocity of $24m/s$. }}
         \label{fig::acc_vel}
\end{figure}

\begin{figure}[h!]
\centering
         \includegraphics[scale=0.5]{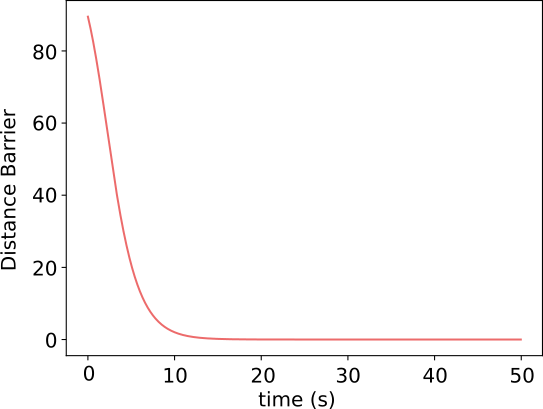}
         \caption{\small{The variation of barrier $h_1$ with time for the ACC problem. The ego agent approaches the minimum allowed distance and maintains it.}}
         \label{fig::acc_barrier}
\end{figure}

\begin{figure}[h!]
\centering
         \includegraphics[scale=0.5]{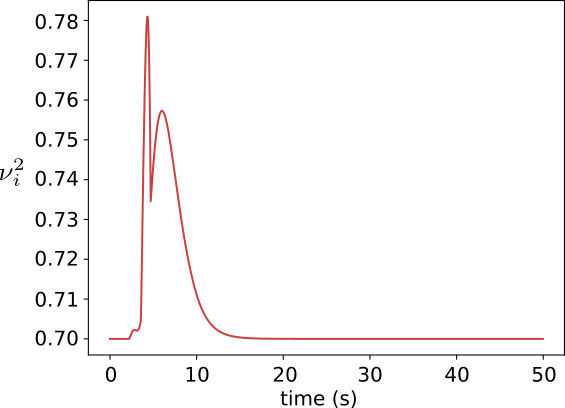}
         \caption{\small{The variation of parameter $\nu_1^2$ with time for the ACC problem. $\nu_i^1$ starts increasing as control bounds are approached. It comes back to its nominal value of $0.7$ eventually because of the $\nu^1_{1,d}$ in \eqref{eq::sampled_data_rule}}}.
         \label{fig::acc_param}
\end{figure}

\begin{figure}[h!]
\centering
         \includegraphics[scale=0.5]{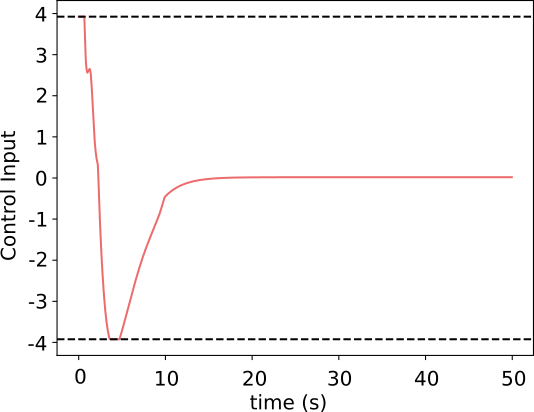}
         \caption{\small{The variation of normalized control input with time for the ACC problem. Normalization is done as $v/M^2$. }}
         \label{fig::acc_control}
\end{figure}

\subsection{Decentralized Multi-Robot Control}
Videos and code available on our website.
We consider a decentralized multi-robot system. The behavior of the neighbors of each agent can be either cooperative or non-cooperative as defined in the earlier section. For the simulation case study, we consider the non-cooperative neighbors of each agent (called ego agent) to be of two behaviors, namely: 1) adversarial; if a neighbor is adversarial, then it means that it is simulated to move head-on towards the ego agent; and 2) uncooperative; if a neighbor is uncooperative then it means that it is simulated to move along an arbitrary direction (resembling a dynamic obstacle). We provide results for two scenarios as shown in Fig. \ref{fig::multi_agent_scenarios}. The initial locations of each ego agent are marked with green circles, uncooperative agents with blue circles, and adversarial agents with red circles. Note that since an agent can be adversarial with respect to some agent but uncooperative to some other, we clarify later on in each case study which agent we are referring to. 

For scenario 1, there are three ego agents $N=1,2,3$ that are modeled as unicycles; note however that none of them assumes that the other ego agents are cooperative. The agents $1,2,3$ use the first-order barrier functions from \cite{wu2016safety} in order to enforce safety. The agent depicted in blue (Agent 5) is adversarial with respect to Agent 1, and is modeled as a single integrator, and the agent in red (Agent 4) is uncooperative (w.r.t. all agents) and moves to the left at constant velocity. None of the ego agents know the identities of any other agent in the system, including other ego agents. The initial values of $\nu_{\alpha_{ij}}$ are all chosen to be 0.8 and the resulting trajectories and variations in barrier function value and $\classK$ function parameters are shown in Figs. \ref{fig::scenario1} and \ref{fig::scenario1_plots}. 

We also compare our results with fixed parameter case. No solution exists to QP\eqref{eq::HOCBF-CLF-QP} after some time when all $\nu_{alpha_{ij}}$ are initialized to 0.8. Hence, assuming identities are known, we choose $\theta_{\alpha_{ij}}$ 2.0 when $j$ is an ego agent and 0.8 otherwise. The fixed parameter CBF controller fails to lead the robots to their goal locations because of the conservative response wherein it steers the ego agents too far away from other agents to ensure safety. In the RT-CBF implementation on other hand, the value of parameters increases when the trust is positive, that is when it is easy to satisfy the CBF constraint, as shown in Fig.\ref{fig::scenario1_plots}. The parameters decrease when trust is negative which happens when an agent is moving too fast towards the ego agent or is very close. Specifically, when agent 1 deviates to the right ($t=0$ to $1.5$s), it increases the trust for agents 2,3 and one of the noncooperative agents as they do not seem to chase after agent 1 and therefore leads to positive trust. Whereas since the adversary and another uncooperative agent move towards agent 1, their trust is 0 or a small negative and does not affect the parameters much. Next, when agent 1 turns back to the left, it perceives a negative trust for agents 2,3 as they seem to be moving towards agent 1. Eventually, since the adversary slows down as it approaches agent 1 (the adversary's controller is based on Lyapunov function) and because the uncooperative agents steer far away, agent 1 is more capable of evading the adversary and therefore positive trust to it. This leads to the transient increase in \classK parameter (red) for collision avoidance with an adversary in Fig. \ref{fig::scenario1_theta}. In our simulation, the beliefs that each ego agent employs are straight-line motions parallel to $Y$ axis for other ego agents and parallel to $X$ axis for other agents.

In scenario 2, shown in Fig. \ref{fig::scenario2}, the ego agents labeled 2 and 3 are modeled as unicycles and use the first-order barrier function from \cite{wu2016safety}, and the agents labeled as 5,6 are modeled as double integrators and use the second-order barrier function $h={ij}=s^2-s{min}^2$ where $s$ is the distance between agents $i,j$ and $s_{min}$ is minimum allowed distance. The adversarial agent chases the unicycle labeled as 2. The uncooperative agent moves to the left at constant velocity but has a visibility cone that must be avoided by ego-agents. We again choose linear $\classK$ functions with $\theta_{\alpha_{ij}}$ initialized to 0.8 for first order barrier function and $\theta_{\alpha_{ij}^0}$ to 0.4 and $\theta_{\alpha_{ij}^1}$ to 0.8 for second order barrier functions.
The response of the fixed parameter case is again to be very conservative compared to the proposed method. 

\begin{figure}[h!]
    \centering
    \begin{subfigure}[b]{0.5\textwidth}
         \centering
         \includegraphics[scale=0.5]{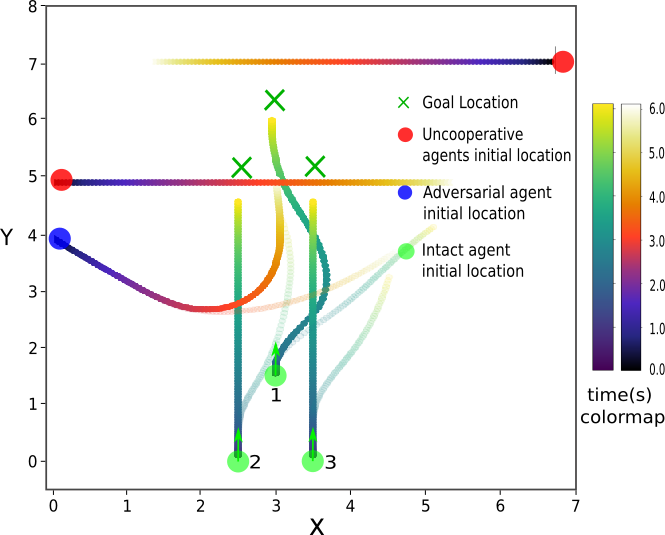}
         \caption{\small{Scenario 1}}
         \label{fig::scenario1}
     \end{subfigure}
     \begin{subfigure}[b]{0.5\textwidth}
         \centering
         \includegraphics[scale=0.45]{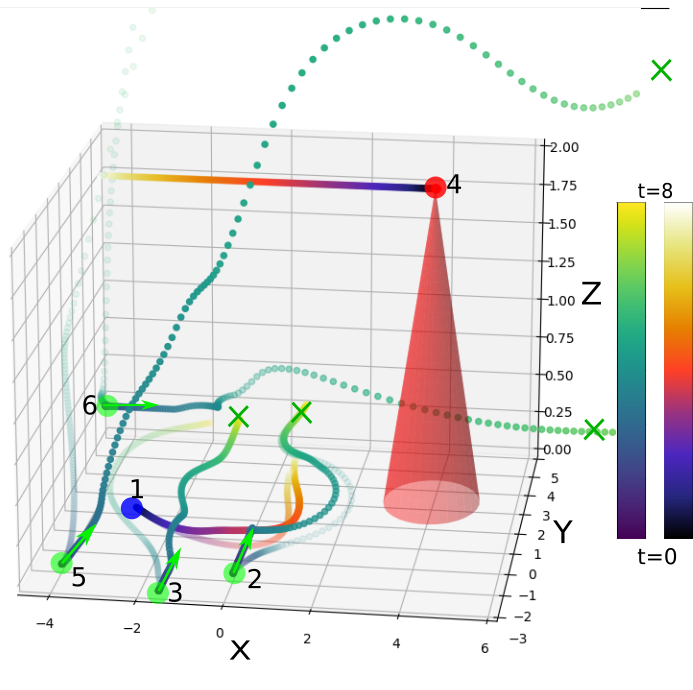}
         \caption{\small{Scenario 2}}
         \label{fig::scenario2}
     \end{subfigure}
    \caption{The timestamp is given by the colormap. Paths with bold colors result from the proposed method. Paths with increased transparency result from the application of CBFs with fixed $\classK$ function parameter.}
    \label{fig::multi_agent_scenarios}
    \vspace{-3mm}
\end{figure}


\begin{figure}[h!]
    \centering
    \begin{subfigure}[b]{0.5\textwidth}
         \centering
         \includegraphics[scale=0.5]{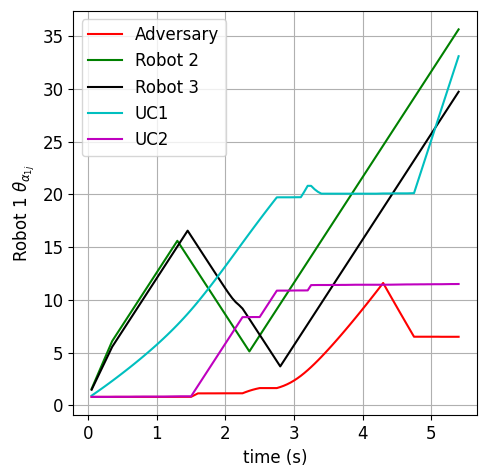}
         \caption{\small{Scenario 1 Parameter Variation}}
         \label{fig::scenario1_theta}
     \end{subfigure}
     \begin{subfigure}[b]{0.5\textwidth}
         \centering
         \includegraphics[scale=0.45]{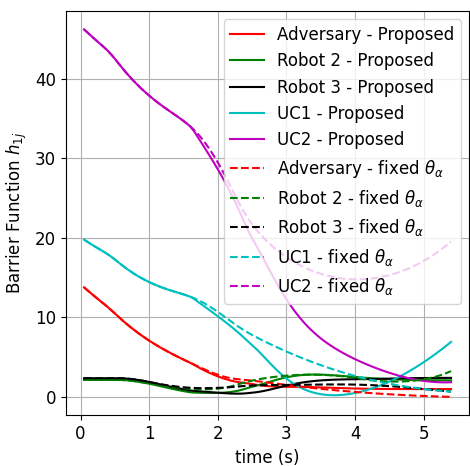}
         \caption{\small{Scenario 1 Barriers}}
         \label{fig::scenario1_barriers}
     \end{subfigure}
        \caption{\small{Scenario 1: Variation of $\classK$ function parameters (a) and barrier functions (b) of Robot 1 with time. UC1, UC2 are the uncooperative agents starting from left and right respectively in Fig. \ref{fig::scenario1}. Trust-based relaxation allows agents to go closer to the safe set boundary ($h_{ij}=0$) compared to fixed $\alpha$ case, and hence leads to a less conservative response while still guaranteeing safety. }}
    \label{fig::scenario1_plots}
    \vspace{-3mm}
\end{figure}

\section{Conclusion}
\label{section::conclusion}
We introduced a new notion of Rate Tunable CBFs that allows for online tuning of CBF-based controllers by adapting the $\classK$ function parameters. We introduce the notion of designing parameter dynamics to enforce compatibility between multiple barrier constraints so that a solution to QP exists. Keeping barrier functions separate rather than combining them into a single barrier function allowed us to shape responses to each of them separately. Finally, we proposed an algorithm to derive pointwise sufficient conditions on \classK parameters for ensuring compatibility of constraints. Furthermore, we present a case study on multi-agent systems wherein we introduce \textit{trust metrics} to design update laws for $\classK$ function parameters. Future work includes analyzing the continuity properties of CBF-QP controllers and exploring if time-varying parameters can be used to improve the continuity of these controllers. We would also like to devise methods based on long-horizon prediction to design optimal update laws for parameters and improve upon the myopic nature of CBF-QP controllers. 




\bibliographystyle{IEEEtran}
\bibliography{journal1.bib}

\begin{thebibliography}{10}
\providecommand{\url}[1]{#1}
\csname url@samestyle\endcsname
\providecommand{\newblock}{\relax}
\providecommand{\bibinfo}[2]{#2}
\providecommand{\BIBentrySTDinterwordspacing}{\spaceskip=0pt\relax}
\providecommand{\BIBentryALTinterwordstretchfactor}{4}
\providecommand{\BIBentryALTinterwordspacing}{\spaceskip=\fontdimen2\font plus
\BIBentryALTinterwordstretchfactor\fontdimen3\font minus
  \fontdimen4\font\relax}
\providecommand{\BIBforeignlanguage}[2]{{%
\expandafter\ifx\csname l@#1\endcsname\relax
\typeout{** WARNING: IEEEtran.bst: No hyphenation pattern has been}%
\typeout{** loaded for the language `#1'. Using the pattern for}%
\typeout{** the default language instead.}%
\else
\language=\csname l@#1\endcsname
\fi
#2}}
\providecommand{\BIBdecl}{\relax}
\BIBdecl

\bibitem{ames2016control}
A.~D. Ames, X.~Xu, J.~W. Grizzle, and P.~Tabuada, ``Control barrier function
  based quadratic programs for safety critical systems,'' \emph{IEEE
  Transactions on Automatic Control}, vol.~62, no.~8, pp. 3861--3876, 2016.

\bibitem{ames2019control}
A.~D. Ames, S.~Coogan, M.~Egerstedt, G.~Notomista, K.~Sreenath, and P.~Tabuada,
  ``Control barrier functions: Theory and applications,'' in \emph{2019 18th
  European control conference (ECC)}.\hskip 1em plus 0.5em minus 0.4em\relax
  IEEE, 2019, pp. 3420--3431.

\bibitem{dawson2023safe}
C.~Dawson, S.~Gao, and C.~Fan, ``Safe control with learned certificates: A
  survey of neural lyapunov, barrier, and contraction methods for robotics and
  control,'' \emph{IEEE Transactions on Robotics}, 2023.

\bibitem{jagtap2020control}
P.~Jagtap, G.~J. Pappas, and M.~Zamani, ``Control barrier functions for unknown
  nonlinear systems using \textrm{G}aussian processes.''\hskip 1em plus 0.5em
  minus 0.4em\relax IEEE, 2020, pp. 3699--3704.

\bibitem{glotfelter2017nonsmooth}
P.~Glotfelter, J.~Cort{\'e}s, and M.~Egerstedt, ``Nonsmooth barrier functions
  with applications to multi-robot systems,'' \emph{IEEE control systems
  letters}, vol.~1, no.~2, pp. 310--315, 2017.

\bibitem{stipanovic2012monotone}
D.~M. Stipanovi{\'c}, C.~J. Tomlin, and G.~Leitmann, ``Monotone approximations
  of minimum and maximum functions and multi-objective problems,''
  \emph{Applied Mathematics \& Optimization}, vol.~66, no.~3, pp. 455--473,
  2012.

\bibitem{panagou2013multi}
D.~Panagou, D.~M. Stipanovi{\v{c}}, and P.~G. Voulgaris, ``Multi-objective
  control for multi-agent systems using \textrm{L}yapunov-like barrier
  functions,'' in \emph{52nd IEEE Conference on Decision and Control}.\hskip
  1em plus 0.5em minus 0.4em\relax IEEE, 2013, pp. 1478--1483.

\bibitem{usevitch2020strong}
J.~Usevitch, K.~Garg, and D.~Panagou, ``Strong invariance using control barrier
  functions: A clarke tangent cone approach,'' in \emph{2020 59th IEEE
  Conference on Decision and Control}.\hskip 1em plus 0.5em minus 0.4em\relax
  IEEE, 2020, pp. 2044--2049.

\bibitem{xiao2020feasibility}
W.~Xiao, C.~A. Belta, and C.~G. Cassandras, ``Feasibility-guided learning for
  constrained optimal control problems,'' in \emph{2020 59th IEEE Conference on
  Decision and Control (CDC)}.\hskip 1em plus 0.5em minus 0.4em\relax IEEE,
  2020, pp. 1896--1901.

\bibitem{wang2022ensuring}
X.~Wang, ``Ensuring safety of learning-based motion planners using control
  barrier functions,'' \emph{IEEE Robotics and Automation Letters}, vol.~7,
  no.~2, pp. 4773--4780, 2022.

\bibitem{nguyen2016exponential}
Q.~Nguyen and K.~Sreenath, ``Exponential control barrier functions for
  enforcing high relative-degree safety-critical constraints,'' in \emph{2016
  American Control Conference (ACC)}.\hskip 1em plus 0.5em minus 0.4em\relax
  IEEE, 2016, pp. 322--328.

\bibitem{xiao2019control}
W.~Xiao and C.~Belta, ``Control barrier functions for systems with high
  relative degree,'' in \emph{2019 IEEE 58th conference on decision and control
  (CDC)}.\hskip 1em plus 0.5em minus 0.4em\relax IEEE, 2019, pp. 474--479.

\bibitem{taylor2022safe}
A.~J. Taylor, P.~Ong, T.~G. Molnar, and A.~D. Ames, ``Safe backstepping with
  control barrier functions,'' \emph{arXiv preprint arXiv:2204.00653}, 2022.

\bibitem{agrawal2021safe}
D.~R. Agrawal and D.~Panagou, ``Safe control synthesis via input constrained
  control barrier functions,'' in \emph{2021 60th IEEE Conference on Decision
  and Control (CDC)}.\hskip 1em plus 0.5em minus 0.4em\relax IEEE, 2021, pp.
  6113--6118.

\bibitem{robey2020learning}
A.~Robey, H.~Hu, L.~Lindemann, H.~Zhang, D.~V. Dimarogonas, S.~Tu, and
  N.~Matni, ``Learning control barrier functions from expert demonstrations,''
  in \emph{2020 59th IEEE Conference on Decision and Control (CDC)}.\hskip 1em
  plus 0.5em minus 0.4em\relax IEEE, 2020, pp. 3717--3724.

\bibitem{lindemann2021learning}
L.~Lindemann, H.~Hu, A.~Robey, H.~Zhang, D.~Dimarogonas, S.~Tu, and N.~Matni,
  ``Learning hybrid control barrier functions from data,'' in \emph{Conference
  on Robot Learning}.\hskip 1em plus 0.5em minus 0.4em\relax PMLR, 2021, pp.
  1351--1370.

\bibitem{dawson2022safe}
C.~Dawson, Z.~Qin, S.~Gao, and C.~Fan, ``Safe nonlinear control using robust
  neural lyapunov-barrier functions,'' in \emph{Conference on Robot
  Learning}.\hskip 1em plus 0.5em minus 0.4em\relax PMLR, 2022, pp. 1724--1735.

\bibitem{black2022adaptation}
M.~Black and D.~Panagou, ``Adaptation for validation of a consolidated control
  barrier function based control synthesis,'' \emph{arXiv preprint
  arXiv:2209.08170}, 2022.

\bibitem{tan2022compatibility}
X.~Tan and D.~V. Dimarogonas, ``Compatibility checking of multiple control
  barrier functions for input constrained systems,'' \emph{arXiv preprint
  arXiv:2209.02284}, 2022.

\bibitem{lindemann2019control}
L.~Lindemann and D.~V. Dimarogonas, ``Control barrier functions for multi-agent
  systems under conflicting local signal temporal logic tasks,'' \emph{IEEE
  control systems letters}, vol.~3, no.~3, pp. 757--762, 2019.

\bibitem{lindemann2020barrier}
------, ``Barrier function based collaborative control of multiple robots under
  signal temporal logic tasks,'' \emph{IEEE Transactions on Control of Network
  Systems}, vol.~7, no.~4, pp. 1916--1928, 2020.

\bibitem{borrmann2015control}
U.~Borrmann, L.~Wang, A.~D. Ames, and M.~Egerstedt, ``Control barrier
  certificates for safe swarm behavior,'' \emph{IFAC-PapersOnLine}, vol.~48,
  no.~27, pp. 68--73, 2015.

\bibitem{usevitch2021adversarial}
J.~Usevitch and D.~Panagou, ``Adversarial resilience for sampled-data systems
  using control barrier function methods,'' in \emph{2021 American Control
  Conference (ACC)}.\hskip 1em plus 0.5em minus 0.4em\relax IEEE, 2021, pp.
  758--763.

\bibitem{mustafa2022adversary}
A.~Mustafa and D.~Panagou, ``Adversary detection and resilient control for
  multi-agent systems,'' \emph{IEEE Transactions on Control of Network
  Systems}, 2022.

\bibitem{zeng2021safety}
J.~Zeng, B.~Zhang, Z.~Li, and K.~Sreenath, ``Safety-critical control using
  optimal-decay control barrier function with guaranteed point-wise
  feasibility,'' in \emph{2021 American Control Conference (ACC)}.\hskip 1em
  plus 0.5em minus 0.4em\relax IEEE, 2021, pp. 3856--3863.

\bibitem{garg2019control}
K.~Garg and D.~Panagou, ``Control \textrm{L}yapunov and control barrier
  functions based quadratic program for spatio-temporal specifications,'' in
  \emph{2019 IEEE 58th Conference on Decision and Control (CDC)}.\hskip 1em
  plus 0.5em minus 0.4em\relax IEEE, 2019, pp. 1422--1429.

\bibitem{seo2022safety}
J.~Seo, J.~Lee, E.~Baek, R.~Horowitz, and J.~Choi, ``Safety-critical control
  with nonaffine control inputs via a relaxed control barrier function for an
  autonomous vehicle,'' \emph{IEEE Robotics and Automation Letters}, vol.~7,
  no.~2, pp. 1944--1951, 2022.

\bibitem{ma2022learning}
H.~Ma, B.~Zhang, M.~Tomizuka, and K.~Sreenath, ``Learning differentiable
  safety-critical control using control barrier functions for generalization to
  novel environments,'' \emph{arXiv preprint arXiv:2201.01347}, 2022.

\bibitem{parwana2022trust}
H.~Parwana, A.~Mustafa, and D.~Panagou, ``Trust-based rate-tunable control
  barrier functions for non-cooperative multi-agent systems,'' in \emph{2022
  IEEE 61st Conference on Decision and Control (CDC)}.\hskip 1em plus 0.5em
  minus 0.4em\relax IEEE, 2022, pp. 2222--2229.

\bibitem{xiao2021adaptive}
W.~Xiao, C.~Belta, and C.~G. Cassandras, ``Adaptive control barrier
  functions,'' \emph{IEEE Transactions on Automatic Control}, vol.~67, no.~5,
  pp. 2267--2281, 2021.

\bibitem{carja2007viability}
O.~C{\^a}rja, M.~Necula, and I.~I. Vrabie, \emph{Viability, invariance and
  applications}.\hskip 1em plus 0.5em minus 0.4em\relax Elsevier, 2007.

\bibitem{lindemann2018control}
L.~Lindemann and D.~V. Dimarogonas, ``Control barrier functions for signal
  temporal logic tasks,'' \emph{IEEE control systems letters}, vol.~3, no.~1,
  pp. 96--101, 2018.

\bibitem{tan2021high}
X.~Tan, W.~S. Cortez, and D.~V. Dimarogonas, ``High-order barrier functions:
  Robustness, safety, and performance-critical control,'' \emph{IEEE
  Transactions on Automatic Control}, vol.~67, no.~6, pp. 3021--3028, 2021.

\bibitem{xiao2021high}
W.~Xiao and C.~Belta, ``High-order control barrier functions,'' \emph{IEEE
  Transactions on Automatic Control}, vol.~67, no.~7, pp. 3655--3662, 2021.

\bibitem{khalil2002nonlinear}
H.~K. Khalil, ``Nonlinear systems third edition,'' \emph{Patience Hall}, vol.
  115, 2002.

\bibitem{qin2021learning}
Z.~Qin, K.~Zhang, Y.~Chen, J.~Chen, and C.~Fan, ``Learning safe multi-agent
  control with decentralized neural barrier certificates,'' \emph{arXiv
  preprint arXiv:2101.05436}, 2021.

\bibitem{parwana2022recursive}
H.~Parwana and D.~Panagou, ``Recursive feasibility guided optimal parameter
  adaptation of differential convex optimization policies for safety-critical
  systems,'' in \emph{2022 International Conference on Robotics and Automation
  (ICRA)}.\hskip 1em plus 0.5em minus 0.4em\relax IEEE, 2022, pp. 6807--6813.

\bibitem{zeng2021enhancing}
J.~Zeng, Z.~Li, and K.~Sreenath, ``Enhancing feasibility and safety of
  nonlinear model predictive control with discrete-time control barrier
  functions,'' in \emph{2021 60th IEEE Conference on Decision and Control
  (CDC)}.\hskip 1em plus 0.5em minus 0.4em\relax IEEE, 2021, pp. 6137--6144.

\bibitem{garg2022control}
K.~Garg, R.~G. Sanfelice, and A.~A. Cardenas, ``Control barrier function-based
  attack-recovery with provable guarantees,'' in \emph{2022 IEEE 61st
  Conference on Decision and Control (CDC)}.\hskip 1em plus 0.5em minus
  0.4em\relax IEEE, 2022, pp. 4808--4813.

\bibitem{diamond2016cvxpy}
S.~Diamond and S.~Boyd, ``Cvxpy: A python-embedded modeling language for convex
  optimization,'' \emph{The Journal of Machine Learning Research}, vol.~17,
  no.~1, pp. 2909--2913, 2016.

\bibitem{agrawal2019differentiating}
A.~Agrawal, S.~Barratt, S.~Boyd, E.~Busseti, and W.~M. Moursi,
  ``Differentiating through a cone program,'' \emph{J. Appl. Numer. Optim},
  vol.~1, no.~2, pp. 107--115, 2019.

\bibitem{wu2016safety}
G.~Wu and K.~Sreenath, ``Safety-critical control of a planar quadrotor,'' in
  \emph{2016 American control conference (ACC)}.\hskip 1em plus 0.5em minus
  0.4em\relax IEEE, 2016, pp. 2252--2258.

\end{thebibliography}

\begin{IEEEbiography}[{\includegraphics[width=1in,height=1.25in,clip,keepaspectratio]{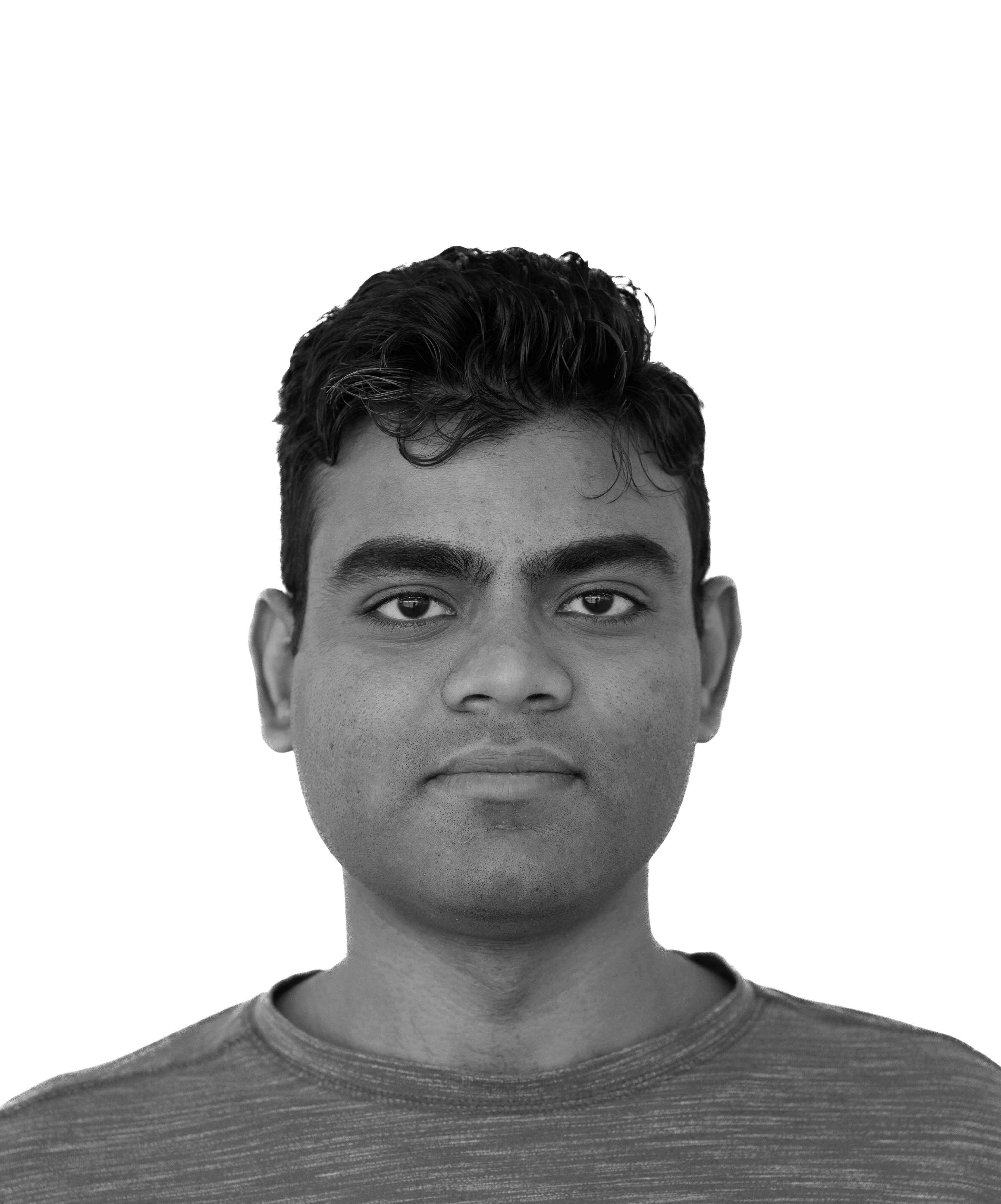}}]{Hardik Parwana} (M'22) received his B.Tech degree in Aerospace Engineering from Indian Institute of Technology Kanpur, India, in 2017, and M.Eng degree in Mechanical Engineering and Science from Kyoto University, Japan, in 2020. 

He is currently a Ph.D. student in the Department of Robotics, University of Michigan at Ann Arbor, MI, USA. His research interests are in safe control and learning for autonomous systems.
\end{IEEEbiography}

\begin{IEEEbiography}[{\includegraphics[width=1in,height=1.25in,clip,keepaspectratio]{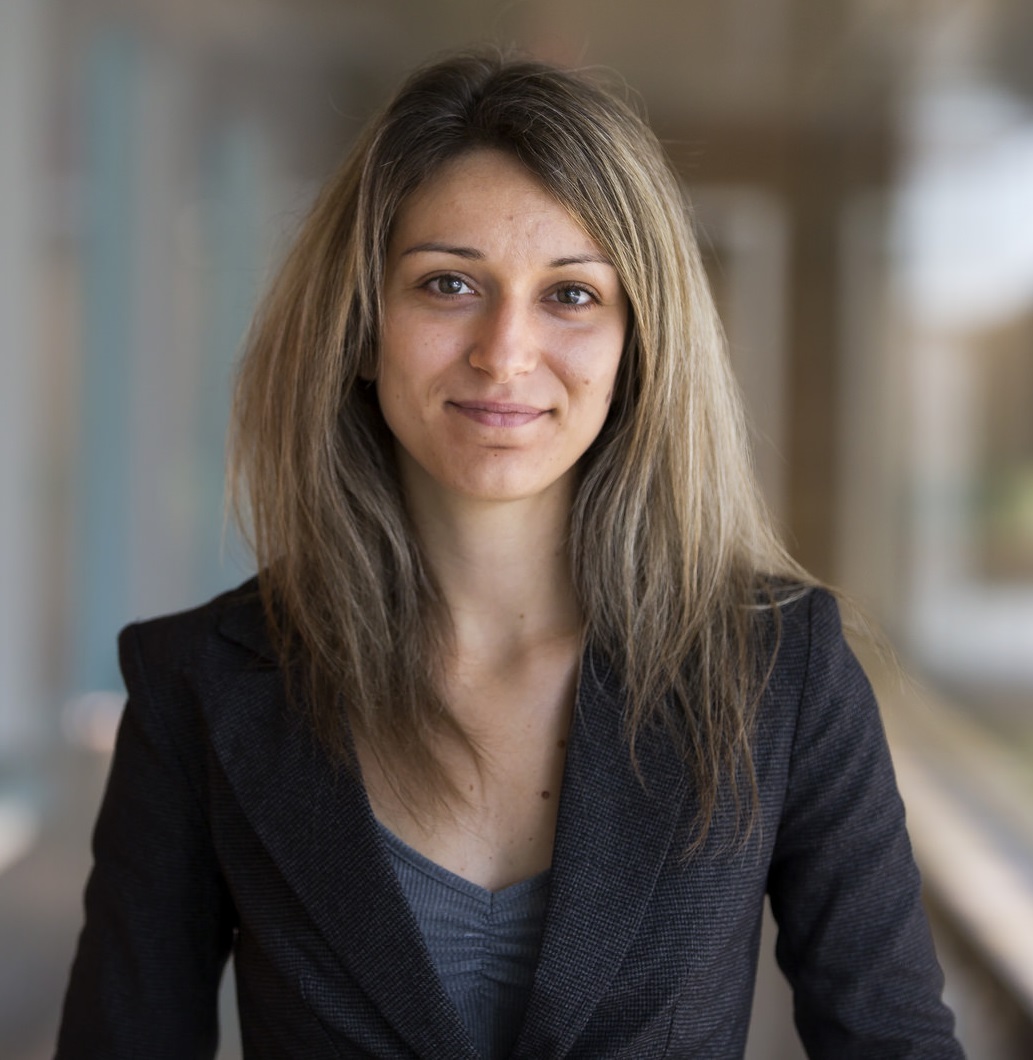}}]{Dimitra Panagou} received the Diploma and PhD degrees in Mechanical Engineering from the National Technical University of Athens, Greece, in 2006 and 2012, respectively. In September 2014 she joined the Department of Aerospace Engineering, University of Michigan as an Assistant Professor. Since July 2022 she is an Associate Professor with the newly established Department of Robotics, with a courtesy appointment with the Department of Aerospace Engineering, University of Michigan. Prior to joining the University of Michigan, she was a postdoctoral research associate with the Coordinated Science Laboratory, University of Illinois, Urbana-Champaign (2012-2014), a visiting research scholar with the GRASP Lab, University of Pennsylvania (June 2013, Fall 2010) and a visiting research scholar with the University of Delaware, Mechanical Engineering Department (Spring 2009). Her research program spans the areas of nonlinear systems and control; multi-agent systems and networks; motion and path planning; human-robot interaction; navigation, guidance, and control of aerospace vehicles. She is particularly interested in the development of provably-correct methods for the safe and secure (resilient) operation of autonomous systems in complex missions, with applications in robot/sensor networks and multi-vehicle systems (ground, marine, aerial, space). She is a recipient of the NASA Early Career Faculty Award, the AFOSR Young Investigator Award, the NSF CAREER Award, and a Senior Member of the IEEE and the AIAA.

\end{IEEEbiography}

\appendix

\subsection{Proof of analytical solution of QP for Theorem 3}
\label{appendix::analytical_qp_soln}

Here we provide a derivation of the CBF-QP controller \eqref{eq::cbf_qp_simple}. The procedure is similar to the one done in \cite{ames2016control} but simplified as we ignore the Lyapunov equation. Consider the following CBF-QP
\begin{subequations}
    \eqn{
        \min_u \quad &  ||u-u_{r}(t,x)||^2 \\
        \textrm{s.t.} \quad &  \dot h(t,x,u) \geq -\alpha_\theta(h(t,x))
        }
\end{subequations}

The above can be written in generic form as
\begin{subequations}
\eqn{
\min_u \quad & u^T Q u + G(t,x)^T u \label{eq::qp_generic_objective}\\
\textrm{s.t.} \quad & A(t,x) u \geq b(t,x) 
}
\label{eq::qp_generic_form}
\end{subequations}

where $Q\in S_m, G\in \reals^{m\times 1}, A\in \reals^{1\times m}, b\in \reals$. The solution of the above QP can be found analytically. Substituting $u = \sqrt{Q}^{-1} u' + \bar u$, where $\bar u = Q^{-1}G$ is the point that minimizes the objective \eqref{eq::qp_generic_objective}, we can reduce \eqref{eq::qp_generic_form} to the following minimum norm problem
\begin{subequations}
    \eqn{
      \min_{u'} \quad & ||u'||^2 \\
      \textrm{s.t.} \quad & A \sqrt{Q}^{-1} u' \geq b + AQ^{-1}G
    }
\end{subequations}

The solution is thus
\eqn{
u = \left\{ 
    \begin{array}{cc}
        Q^{-1} \Bigg[ A^T \frac{b + AQ^{-1}G}{||A\sqrt{Q}^{-1}||} - G\Bigg] & \textrm{if } b+AQ^{-1}G \geq 0 \\
        -Q^{-1}G & \textrm{if } b+AQ^{-1}G \leq 0
    \end{array}
\right.
}
Substituting for $A = L_gh, b=-\alpha(h)-L_f h, Q = 2I, G = -2u_r$, we get
\eqn{
u_{QP} = \left\{ 
    \begin{array}{cl}
        u_r - \frac{1}{\sqrt{2}}\frac{L_gh^T}{||L_gh||}(L_fh+L_gh u_r+ \alpha(h)) & \textrm{if } \psi^1 < 0 \\
        u_r & \textrm{if } \psi^1 \geq 0
    \end{array}
\right.
}
where $\psi^1 = L_f(x) h(t,x) + L_g(x) h(t,x) u_r+ \alpha_\theta(h(t,x))$.

Now suppose we would like the solution of this optimization, hereby referred to as $u_{QP}$ to match a safe control input $u_d$. Further, suppose that $\psi^1 < 0$. Then we have
\eqn{
u_d - u_{QP} = u_d - u_r + \frac{1}{\sqrt{2}} \frac{L_gh^T}{||L_gh||}(L_fh+L_gh u_r+ \alpha(h))
}
Note that $u_d - u_{QP}=0$ only if $u_d-u_r$ is in the span of the vector $L_gh^T$. This equation serves to show two trivial observations: First, despite the existence of a CBF being a necessary condition for safety\cite{ames2016control}, the \textit{myopic} CBF-QP controller is not complete, that is, it cannot result in every possible safe trajectory of the system. Note that existing works to improve the optimality of CBFs can also be attributed to making them more complete.

\end{document}